\documentclass[11pt,english]{article}
\usepackage{amsmath,amsthm,url,amsfonts,amssymb,amscd}
\usepackage{enumerate}
\usepackage{hyperref}
\usepackage[latin1]{inputenc}

\begin{document}
\theoremstyle{plain}
\newtheorem{theo}{Theorem}[section]
\newtheorem{thm}[theo]{Theorem}
\newtheorem{prop}[theo]{Proposition}
\newtheorem{coro}[theo]{Corollary}
\newtheorem{conj}{Conjecture}
\newtheorem{propr}{Property}[subsection]
\newtheorem{lemm}[propr]{Lemma}
\newtheorem{lem}[propr]{Lemma}
\newtheorem{claim}[propr]{Claim}
\newtheorem{proprf}[propr]{Fundamental property}
\theoremstyle{definition}
\newtheorem{exem}[propr]{Example}
\newtheorem{exems}[propr]{Examples}
\newtheorem{rema}[propr]{Remark}
\newtheorem{defi}[propr]{Definition}
\newtheorem{ques}[propr]{Question}
\newtheorem{Ques}[propr]{Question}
\setcounter{tocdepth}{1}
\title{Structural stability of attractor-repellor endomorphisms with singularities}
\date {}
\author{Pierre Berger}
\maketitle
\begin{abstract}
We prove a theorem on structural stability of smooth attractor-repellor endomorphisms of compact manifolds, with singularities. 
By attractor-repellor, we mean that the non-wandering set of the dynamics $f$ is the disjoint union of a repulsive compact subset with a hyperbolic attractor on which $f$ acts bijectively.
The statement of this result is both infinitesimal and dynamical. Up to our knowledge, this is the first in this hybrid direction.
Our results generalize also a Mather's theorem in singularity theory which states that infinitesimal stability implies structural stability for composed mappings, to the larger category of laminations.
\end{abstract}
\setlength{\parskip}{0pt}
\tableofcontents
\section{Motivations and statement of the main result}
\subsection{Structural stability of Dynamical Systems}
Let $f$ be a \emph{smooth endomorphism} of a compact manifold $M$. This means that $f$ is a $C^\infty$ map from $M$ into itself which is not necessarily bijective and can have \emph{singularities}: its tangent map can be non surjective at some points. For $k\in \mathbb N\cup \{\infty \}$, the map $f$ is \emph{$C^k$-structurally stable}, if for every $f'$ close to $f$ in the $C^k$-topology, there exists a homeomorphism $h$ of $M$ 
such that the following diagram commutes:
\[\begin{array}{rcccl}
& &f' & &\\ 
&M & \rightarrow&M &\\
h&\uparrow& &\uparrow&h\\
&M & \rightarrow&M &\\
& &f & &\\ \end{array}\]
The topology of the space of $C^k$ maps is the usual one for $k<\infty$, and for $k=\infty$ it is the union of the $C^j$ topologies for $j\in \mathbb N$.
The combined work of many authors (Smale \cite{Sm}, Palis \cite{PS}, De Melo \cite{dM}, Robin \cite{Ri}, Robinson \cite{Rs}, Ma\~ne \cite{Mane})
has led to a satisfactory characterization of $C^1$-structurally stable diffeomorphisms. Let us state their characterization:
\begin{thm}
Let $M$ be a compact connected manifold and $f$ a $C^1$ diffeomorphism of $M$. Then the following conditions are equivalent:
\begin{enumerate}
\item 
$f$ is $C^1$-structurally stable,
\item $f$ satisfies Axiom $A$ and the strong transversality condition,
\item $f$ is $C^0$-infinitesimally stable.
\end{enumerate}
\end{thm}
Let us recall the definitions of statements 2 and 3.
A diffeomorphism $f$ satisfies \emph{Axiom $A$} if its non-wandering set 
 $\Omega$ is hyperbolic and equal to the closure of the set of periodic points.
The diffeomorphism $f$ satisfies moreover the \emph{strong transversality condition} if the stable and unstable manifolds of points of $\Omega$ intersect each other transversally.

The concept of infinitesimal stability will play a crucial role in this work. A $C^{r+1}$ endomorphism is \emph{$C^r$-infinitesimally stable} if the following map is surjective:
\[\sigma\in \chi^r(M) \mapsto \sigma\circ f-Tf\circ \sigma\in \chi^r(f),\]
with $\chi^r(M)$ the space of $C^r$ sections of the tangent bundle $TM$, and with $\chi^r(f)$ the space of $C^r$ sections of the push forward bundle $f^*TM$ whose fiber at $x$ is $T_{f(x)}M$.

 In order to understand infinitesimal stability, let us regard a smooth endomorphism $f$ of the torus $\mathbb T=\mathbb R^n/\mathbb Z^n$. We notice that the following map is Fr\'echet differentiable:
\[\phi\;:\; Diff^\infty(\mathbb T)\times Diff^r(\mathbb T)\rightarrow C^r(\mathbb T,\mathbb T)\]
\[(f', h)\mapsto h\circ f'-f'\circ h,\]
where $Diff^\infty(\mathbb T)$ and $Diff^r(\mathbb T)$ are the spaces of $C^\infty$ and $C^r$ diffeomorphisms of $\mathbb T$ respectively.
 Moreover, its partial derivative at $(f,id)$ with respect to the second variable is 
\[\sigma\in \chi^r(\mathbb T)\mapsto \sigma\circ f-Tf\circ \sigma\in \chi^r(\mathbb T).\]

Consequently the above partial derivative is surjective iff $f$ is $C^r$-infinitesimally stable.

\subsection{Structural stability in Singularity Theory}
Meanwhile, the school initiated by Whitney and Thom was interested in the following problem.
Let $f$ be a smooth map from a manifold $M_1$ into a manifold $M_2$. 
For $k> 0$, the map $f$ is \emph{$C^k$-equivalently stable}\footnote{Usually, in singularity theory, we say $C^k$-structurally stable; however this conflicts the dynamical terminology.} if there exists a neighborhood $U$ of $f$ in the $C^\infty$ topology such that for any $f'\in U$, there exist $C^k$ automorphisms $h_1$ and $h_2$ of $M_1$ and $M_2$ respectively such that the following diagram commutes:
\[\begin{array}{rcccl}
& &f' & &\\ 
&M_1 & \rightarrow&M_2 &\\
h_1&\uparrow& &\uparrow&h_2\\
&M_1 & \rightarrow&M_2 &\\
& &f & &\\ \end{array}.\]
The problem is then to describe the $C^k$-equivalently stable maps which is usually simpler. This program was carried out by Mather who solved many  conjectures of Thom. One of his results is the following:
\begin{thm}[Mather \cite{Ma1}, \cite{Ma2}]
The $C^\infty$-equivalently stable, proper maps are the $C^\infty$-equivalent infinitesimal stable maps.\end{thm}
Here \emph{$C^\infty$-equivalent infinitesimal stability} means that the following map is onto:
\[(\sigma_1,\sigma_2)\in \chi^\infty(M_1)\times \chi^\infty(M_2)\mapsto \sigma_2\circ f-Tf\circ \sigma_1 \in \chi^\infty(f).\]
\subsection{Statement of the main result}
Concerning the structural stability of endomorphisms, there is not yet a criterion, neither a satisfactory description. For instance, Dufour \cite{Dufour} showed that there is no $C^\infty$-infinitesimal stable endomorphism with a periodic point. On the other hand, 
the map $x\mapsto x^2$ on the one point compactification of $\mathbb R$, is $C^0$-infinitesimally stable but not $C^\infty$-structurally stable.

There are very few old theorems stating sufficient conditions for the structural stability of endomorphisms that are not diffeomorphisms.  Nowadays the subject regains of interest with new examples (\cite{Rov1}, \cite{Rov2}). We recall below all these theorems. 
Our main theorem generalizes all the theorems implying the $C^\infty$-stability of endomorphisms that are not diffeomorphisms.

A first old result was stated by Shub in his PhD thesis \cite{Shubthese}, and requires the following definition:
\begin{defi} Let $f$ be a $C^1$ endomorphism of a Riemannian manifold $(M,g)$, which sends a compact subset $K\subset M$ into itself. The compact subset $K$ is \emph{repulsive} if there exist $\lambda>1$ and $n\ge 1$ such that for every $x\in K$, the tangent map $T_xf^n$ is invertible with contractive inverse.
The map $f$ is \emph{expanding} if $M$ is compact and $K=M$.
\end{defi}
\begin{thm}[Shub 69] The expanding $C^1$-endomorphisms of a compact manifold are $C^1$-structurally stable.\end{thm}
The following result is known since many times:
\begin{thm} Let $f$ be a rational function of the Riemann sphere. If the Julia set of $f$ is repulsive and if the critical points of $f$ are quadratic, with non-preperiodic and disjoint orbits, then $f$ is structurally stable for holomorphic perturbations.\end{thm}

In the above example the singularities are not (equivalently) stable for smooth perturbations and so the dynamics is not $C^\infty$-structurally stable.
However, since its critical points are necessarily in the attracting basin of periodic attracting orbits, it is a good example of attractor-reppelor dynamics.  

\begin{defi} Let $f$ be a smooth endomorphism of a compact, non necessarily connected manifold. 
The endomorphism $f$ is \emph{attractor-repellor} if its non-wandering set is the disjoint union of two compact subsets $R$ and $A$ such that: 
%a repulsive compact set $R$ with a hyperbolic compact set $A$ such that:
\begin{itemize}\item the compact subset $R$ is repulsive, but non necessarily transitive, %\item the closure of the backward image $\cup_{n\ge 0} f^{-n}(R)$ is a repulsive compact set,
\item the compact subset $A$ is hyperbolic, the restriction of $f$ to $A$ is bijective, and the unstable manifolds of points of $A$ are contained in $A$. However $A$ is not necessarily transitive.\end{itemize}
\end{defi}

From the above discussion we understand that our result on structural stability of endomorphisms needs to mix criteria from dynamical systems and singularity theory. 
\begin{thm}[Main Result]\label{main} Let $f$ be an {attractor-repellor}, smooth endomorphism of a compact, non necessarily connected manifold $M$. If the following conditions are satisfied, then $f$ is $C^\infty$-structurally stable:\begin{itemize}
\item[(i)] the periodic points of $f$ are dense in $A$,
\item[(ii)] the singularities $S$ of $f$ have their orbits that do not intersect the non-wandering set $\Omega$,
\item[(iii)] the restriction of $f$ to $M\setminus \hat \Omega$ is $C^\infty$-infinitesimally stable, with $\hat \Omega:= cl\big(\cup_{n\ge 0} f^{-n}(\Omega)\big)$, 
\item[(iv)] $f$ is transverse to the stable manifold of $A$'s points: for any $y\in A$, for any point $z$ in a local stable manifold $W^s_y$ of $y$, for any $n\ge 0$, and for any $x\in f^{-n}(\{z\})$, we have: 
\[Tf^n(T_xM)+T_zW_y^s=T_zM.\]
\end{itemize} 
\end{thm}
We recall that the \emph{singularities} of a smooth endomorphism $f$ of $M$ are the points $x$ in $M$ for which the tangent map $T_x f$ is not surjective.
\begin{rema}\label{regu} Actually, we prove that the conjugacy map $h$ between $f$ and its perturbation is a smooth immersion restricted to each stable manifold of $f$ without $\hat A:= \cup_{n\ge 0} f^{-n}(A)$. 
Moreover, we prove that the partial derivatives along the stable manifolds depend continuously on the base point over all $W^s(\Omega)\setminus \hat \Omega$. 
\end{rema}
The following non-trivial remark follows easily from the proof of the main result:
\begin{rema} The hypotheses of this theorem are open: any small smooth perturbation of $f$ also satisfies them.\end{rema} 

The main theorem generalizes a well known result:
\begin{coro} \label{dim1}
 Let $f$ be a smooth endomorphism of the circle $\mathbb S^1$. If $f$ is an attractor-repellor endomorphism such that the critical points of $f$ are quadratic, with disjoint orbits in the basin of the attractor but non-preperiodic, then $f$ is $C^\infty$-structurally stable.\end{coro}
\begin{rema} Actually by \cite{Van}, the above corollary is maximal in dimension 1: there are no more structurally stable endomorphisms of the circle. Moreover a $C^\infty$-generic map of the circle satisfies this hypothesis. \end{rema}

\begin{proof}[Proof of Corollary \ref{dim1}]
 Only the infinitesimal condition hypothesis is not obvious. By Lemma \ref{trivial inf stab} (see below), we only need to prove the $C^\infty$-equivalent infinitesimal stability of the map $f:\;x\in \mathbb R\mapsto x^2\in \mathbb R$.  Let $\xi\in \chi(f)$. 
Let $\sigma_2$ be the constant function equal to  $\xi(0)$ and:
\[\sigma_1(x):= \frac{\sigma_2(x^2)-\xi(x)}{2x}=-\int_0^1 \frac{\xi'(tx)}{2}dt\]
   We notice that $\sigma_1,\sigma_2\in \chi^\infty (\mathbb S^1)$ satisfy:
  \[ \sigma_2(x^2)-2x\sigma_1(x)=\xi(x), \quad \forall x\in \mathbb S^1.\]
\[\mathrm{i.e.}\; \sigma_2\circ f-Tf\circ \sigma_1=\xi.\]
\end{proof} 
Equivalent infinitesimal stability is a $C^\infty$-generic property for maps from compact manifolds of dimension less than 9 (see \cite{Gol} p163 when the differentiable map sends a manifold into one of different dimension). Conditions on the dimension of the manifold are given in \cite{Nakai} when the singularities overlap.

This provides other applications of the main result.

\begin{exem} Let $R$ be a rational function of the Riemann sphere such that all its critical points belong to the basin of attracting periodic orbits. Then the non-wandering set $\Omega$ of $R$ split into two sets: the union of the basin of a finite number of attracting periodic points, and the Julia set $J$ which is a repulsive compact subset. Thus $R$ is an attractor-repellor endomorphism of the sphere. But condition $(ii)$ is not satisfied since its singularities are not stable.
Nevertheless, for $C^\infty$-generic perturbations $R'$ of $R$ the singularities $\Sigma$ of $R$ are equivalently stable and do not overlap along their orbits ($f^k(\Sigma)\cap \Sigma=\emptyset$, $\forall k>0$). Thus $R'$ satisfies the hypothesis of the main theorem, and hence is structurally stable. To have the shape of the singularities of $R'$ see \cite{arnold}  p. 20. This result was recently shown in \cite{Rov2}.\end{exem}

\begin{exem} When the endomorphism $f$ does not have singularities, hypothesis $(iii)$ is always satisfied, as we will see in Lemma \ref{trivial inf stab}. Consequently, our main result implies the structural stability of the following endomorphism:
\[f:\; \mathbb S^2\times \mathbb S^1\rightarrow \mathbb S^2\times \mathbb S^1\]
\[(z,z')\mapsto \Big( \frac{z+z'}{3},z'^2\Big),\]
where $\mathbb S^2$ is the Riemann sphere and $\mathbb S^1$ is the unit circle of the complex plane. We notice that the non-wandering set is the disjoint union of the Smale solenoid (inside the product of the unit disk with $\mathbb S^1$) with the repulsive circle $\{\infty\}\times \mathbb S^1$. This example is the first structurally stable endomorphism known  which is neither a diffeomorphism nor expanding and was found in \cite{Rov1}. 
\end{exem}

\begin{exem} Let us introduce some singularities in the previous example. Let $\mathbb D$ be the closed unit disk. We notice that $\mathbb D\times \mathbb S^1$ contains the solenoid, is sent into its interior by the dynamics and the restriction of $f$ to $\mathbb D\times \mathbb S^1$ is a diffeomorphism onto its image. Let us parametrize $\bar{\mathbb T}:= \mathbb D\times \mathbb S^1$ in polar coordinates  $\{(r e^{i\theta},z');\; r\in [0,1],\; \mathrm{and} \; z'\in \mathbb S^1\}$. Let $\rho\in C^\infty ([0,1],[0,1])$ be a Morse function equal to the identity  everywhere except a small subset $U$ close to $1$ but with closure disjoint from $\{1\}$. We suppose that $\rho$ has exactly two singularities and that $\rho$ sends $U$ into $U$. We notice that $\rho_1:\; re^{i\theta}\mapsto \rho(r)e^{i\theta}$ has for singularities two folds over two disjoint circles. 

For some perturbation $\rho_2$ of $\rho_1$, the diffeomorphism $\rho_2$ is still equal to the identity on the complement of $U\times \mathbb S^1$, its singularities consist of two folds over two disjoint circles transverse to those of $\rho_1$, and $\rho_2$ sends $U\times \mathbb S^1$ into itself. Let us now regard:

\[ f:\; (z,z')\mapsto \left\{\begin{array}{cc} f(z,z')& \mathrm{for}\; (z,z')\in (\mathbb S^2\times \mathbb S^1\setminus \bar{\mathbb T})\cup f^2(\bar{\mathbb T})\\
   f(\rho_1(z),z')& \mathrm{for}\; (z,z')\in \bar{\mathbb T}\setminus f(\bar{\mathbb T})\\
f^2\circ (\rho_2(\cdot, id) \circ f^{-1}_{|\bar{\mathbb T}}(z,z')  & \mathrm{for}\; (z,z')\in f(\bar{\mathbb T})\setminus f^2(\bar{\mathbb T})\end{array}\right.\]

Let us show how that the main theorem implies that $f'$ is $C^\infty$-structurally stable. First we notice that $f'$ is a smooth endomorphism of $\mathbb S^2\times \mathbb S^1$, and is an attractor-repellor endomorphism: it has the same non-wandering set as $f$ $(i)$. Moreover $f'$ is transverse to the stable manifolds of the solenoid $(iv)$. Its singularities are formed by four folds; only two of them intersect the other ones along their orbits. This intersection occurs only at the first iteration and is transverse. Since the singularity of $\rho$ are close to $\{1\}$, the singularities of $f'$ are disjoint from the solenoid, but included in $\bar{\mathbb T}$. As $f$ embeds $\bar{\mathbb T}$ into itself, the orbits of the singularities do not intersect the non-wandering set $(ii)$. Let us prove Property $(iii)$. For this end we are going to show the $C^\infty$-equivalent infinitesimal stability of the diagram: 
\[ \begin{array}{ccccc}         &f_1       &           &     f_2   &\\
\mathbb R^2&\rightarrow&\mathbb R^2&\rightarrow&\mathbb R^2\\
   (x,y)&\mapsto       &(x^2,y)    &           &\\
        &              &(x,y)      &\mapsto    &\big((x+y)^2,y\big)\end{array}\]   
        The \emph{$C^\infty$-equivalent infinitesimal stability}  of this diagram means that the map $(\sigma_1,\sigma_2,\sigma_3)\in \chi^\infty(\mathbb R^2)^3\mapsto (\sigma_3\circ f_1-Tf_2\circ \sigma_1,\sigma_3\circ f_2-Tf_3\circ \sigma_1)\in \chi^\infty(f_1)\times \chi^\infty(f_2)$ is surjective. Such an infinitesimal stability will be sufficient for our purpose.

The product of above maps with the identity of $\mathbb R$ provides a local model of the singularities of $f$. The equivalent infinitesimal stability of such a product follows easily from the one of the above diagram. Also by using a partition of the unity, we get the infinitesimal stability of $f$ restricted to a neighborhood of the singularities. Then Lemma \ref{trivial inf stab} implies the infinitesimal stability Property $(iii)$.

Let us prove the $C^\infty$-equivalent infinitesimal stability of the above diagram. 
Let $\xi\in \chi^\infty(f_1)$ and $\zeta\in \chi^\infty(f_2)$. We want to construct $\sigma_1, \sigma_2,\sigma_3\in \chi^\infty(\mathbb R^2)$ such that:
\[\sigma_2\circ f_1-Tf_1(\sigma_1)=\xi\quad \mathrm{and}\quad \sigma_3\circ f_2-Tf_2(\sigma_2)=\zeta.\]
By putting an exponent $1$ or $2$ to denote the first or second component respectively, the above system is equivalent to:

\begin{align}
 \nonumber 	\sigma_2^1(x^2,y)-2x \sigma_1^1(x,y)=\xi^1(x,y)\\
 \nonumber 	\sigma_2^2 (x^2,y)-\sigma_1^2(x,y)=\xi^2(x,y)\\
 \nonumber 	\sigma_3^1\big((x+y)^2,y\big)-2(x+y)\big(\sigma_2^1(x,y)+\sigma_2^2(x,y)\big)=\zeta^1(x,y)\\
 \nonumber 	\sigma_3^2\big((x+y)^2,y\big)-\sigma_2^2(x,y)=\zeta^2(x,y)
\end{align}
 Therefore it is sufficient to solve the following system with $X:= x+y$ and $Y:=y$:
\begin{align}
	\sigma_1^1(x,y)=\frac{\sigma_2^1(x^2,y)-\xi^1(x,y)}{2x}\\
	\sigma_2^1(X,Y)+\sigma_2^2(X,Y)=\frac{\sigma_3^1\big(X^2,Y)-\zeta^1(X,Y)}{2X}\\
	\sigma_1^2(x,y)=\sigma_2^2 (x^2,y)-\zeta^2(x^2,y)\\
	\sigma_2^2(X,Y)=\sigma_3^2\big(X^2,Y\big)-\zeta^2(X,Y)
\end{align}
We put $\sigma_3^1(X,Y):=\zeta(0,Y)$ which makes sense to $(2)$ when $X$ approaches $0$ and
\[\sigma_3^2(Y^2,Y):=-\xi^1(Y,Y)+\frac{\zeta^1(0,Y)-\zeta^1(Y,Y)}{2Y}+\zeta^2(Y,Y)\]
 and smoothly extended off the parabola $X=Y^2$.
 
 Thus $\sigma_2^2$ is defined without ambiguity by $(4)$, $\sigma_2^1$ by $(2)$, and $\sigma_1^2$ by $(3)$.
 
 Finally we compute that $\sigma_2^1(0,y)= \xi^1(0,y)$ and so we can well define $\sigma_1^1$ by $(1)$. 
\end{exem}
\begin{exem}
Let $M$ be the $2$-sphere that we identify with the one-point compactification of $\mathbb R^2$. Let $(r,\theta)$ be the polar coordinates of $\mathbb R^2$. 
\[\mathrm{Let}\quad f:\;\hat {\mathbb R}^2 \rightarrow \hat {\mathbb R}^2\]
\[(r,\theta)\mapsto \Big( \frac{r}{2+2 r^2}, \theta\Big).\]
We notice that $f$ is an attractor-repellor endomorphism with $\{0\}$ as attractor and with empty repellor. Also the singularities of $f$ are folds that do not overlap. Consequently the hypotheses of the main theorem are satisfied and so $f$ is structurally stable.
\end{exem}

Here is the lemma that we needed for all our computations.
\begin{lemm}\label{trivial inf stab} Let $f$ be an attractor-repellor endomorphism of a compact manifold $M$. Suppose that $f$ satisfies Properties $(i)$ and $(ii)$ of the main theorem. Let $\Omega$ be the non-wandering set of $f$ and $M':= M\setminus cl(\cup_{n\ge 0} f^{-n}(\Omega))$. Suppose moreover the existence of an open neighborhood $U\subset M'$ of the singularities such that:
\begin{itemize}
\item[$(a)$] for every $x\in U$, if an iterate $f^n(x)$ belongs to $U$, then $(f^k(x))_{k=0}^n$ belongs to $U$,
\item[$(b)$] The restriction of $f^n$ to $U\setminus f^{-1}(U)$ is injective with injective derivative, for every $n\ge 0$,
\item[$(c)$] the map $\sigma\in \chi^\infty(M')\mapsto (\sigma\circ f- Tf\circ \sigma)_{|U}\in   \chi^\infty(f_{|U})$ is surjective,
\end{itemize}
then the restriction of $f$ to $M'$ is $C^\infty$-infinitesimally stable.
\end{lemm}
This lemma will be shown at this end of in Subsection \ref{final lemma}.

\subsection{Links with structural stability of composed of mappings}
If we mentioned several times that the manifolds are not necessarily connected, it is because in this case, our main result is a complement to a theorem of Mather on structural stability of composed mapping. Moreover, contrarily to what happen for diffeomorphisms or local diffeomorphisms (such as expanding maps), the endomorphisms can send a connected component of the manifold into one of different dimension.
As mentioned by Baas \cite{BAAS}, the problem of composed mapping was first stated by Thom, and have many applications in Biology (see \cite{MR0467804} and \cite{MR0322905}), in the study of network (see Baas \cite{BAAS}) and in the study of the so-called Laudau singularity of Feynman integral (see \cite{MR0214341}).

To state the problem of composed mapping, let us consider a finite oriented graph $G:=(V,A)$ with a manifold $M_i$ associated to each vertex $i\in V$, and with a smooth map $f_{ij}\in C^\infty(M_i,M_j)$ associated to each arrow $[i,j]\in A$ from $i$ to $j$.
\begin{exem}
\[\begin{array}{ccccc}
& & f_{12}\qquad &\\
& M_1 & \rightarrow  M_2&\\
f_{12}&\uparrow&\nearrow f_{32} &\\
& M_3 &&\\
&\circlearrowright& &  \\
& f_{33}&&\end{array}\]
\end{exem}
For $k\ge 0$, such a graph is \emph{$C^k$-equivalently stable} if for every 
$(f'_{ij})_{[i,j]\in A}$ close to $(f_{ij})_{[i,j]\in A}$ in $\prod_{[i,j]\in A}C^\infty(M_i,M_j)$, there exist diffeomorphisms $(h_{i})_i\in \prod_{i\in V} C^k(M_i, M_i)$ s.t. for every $[i,j]\in A$ the following diagram commutes:
\[\begin{array}{rcccl}
& &f'_{ij} & &\\ 
&M_i & \rightarrow&M_j &\\
h_i&\uparrow& &\uparrow&h_j\\
&M_i & \rightarrow&M_j &\\
& &f_{ij} & &\\ \end{array}.\]
The graph $G$ is \emph{convergent} if at most one arrow of $A$ starts from each vertex. The graph is \emph{without cycle} if for any family of arrows $([i_{k} , i_{k+1}])_{k=1}^N\in A^N$, the vertices $i_1$ and $i_{N+1}$ are different.
The following theorem was proved by Mather, and then written by Baas \cite{BAAS}:
\begin{thm}\label{1.7}
Let $G$ be a graph of smooth proper maps, convergent and without cycle. The graph is $C^\infty$-equivalently structurally stable if the following map is surjective:
\[\prod_{i\in V}\chi^\infty(M_i)\rightarrow \prod_{[i, j]\in A}\chi^\infty (f_{ij})\]
\[(\sigma_i)_i\mapsto (Tf_{ij}\circ\sigma_i-\sigma_j\circ f_{ij})_{[ij]}.\]
\end{thm}
We will see that our main theorem generalizes the above result when the manifolds $(M_i)_i$ are compact. Though this theorem has never been published, but it has been cited several times (\cite{Nakai}, \cite{Dufour}, \cite{Michael}...). Hence we will explain how to deduce the proof of this theorem in the non-compact case from this work  (see Remark \ref{rem pour 1.7}).   
But before, we shall notice that there is a canonical graph of maps associated to each smooth endomorphism $f$ of a compact manifold $M$. Let $(M_i)_{i\in V}$ be the connected components of $M$. Let $A$ be the set of arrows $[i,j]$ such that $f$ sends $M_i$ into $M_j$. Let $f_{ij}$ be the restriction of $f$ to $M_i$. Therefore, $G:=(V,A)$ is a graph of smooth maps, convergent but \emph{always with cycles}.
Also, we notice that the $C^\infty$-structural stability of $f$ is the $C^0$-equivalent stability of this graph of maps.

Conversely, let us show that our main theorem implies the Mather's one on equivalent stability of graphs of maps, in the compact case. Let $G=(V,A)$ be a convergent and without cycle graph of maps of compact manifolds. 
Let $\hat V$ be the union of $V$ with the circle $\hat{\mathbb R}$ and with the trivial $0$-dimensional manifold $\{0\}$. We identify the circle $\hat {\mathbb R}$ to the one-point compactification of the real line. Let $f_{\hat{\mathbb R}\hat{\mathbb R}}:=x\in \hat{\mathbb R}\mapsto 2x$ and let $f_{0\hat{ \mathbb R}}:= 0\in \{0\}\mapsto 1\in \hat {\mathbb R}$.

Let $V'\subset V$ be the subset of vertices from which no arrow starts. 
For $i\in V'$, let $f_{i0}$ be the constant map from $M_i$ onto $\{0\}$. Let finally $\hat A:=A\cup \{[i,0]\}_{ i\in V'}\cup\{[0,\hat {\mathbb R}]\} \cup \{[\hat{\mathbb R},\hat{\mathbb R}]\}$. One easily remarks that the hypotheses of Mather's theorem for the graph $G$ implies those of the main result for the smooth endomorphism $f$ of the disjoint union $M=\coprod_{i\in \hat V} M_i$, whose restriction to $M_i$ is the map $f_{ij}$, with $i\in \hat V$ and $[i,j]\in \hat A$. Also for any smooth perturbation $(f'_{ij})_{[i,j]\in A}$ of $(f_{ij})_{[i,j]\in A}$, we can use the above algorithm to associate a perturbation $f'$ of $f$. By the main theorem, the endomorphism $f'$ is conjugated to $f$. By Remark \ref{regu}, the conjugacy is smooth along the stable manifold of $M$. As they contain each manifold $(M_i)_{i\in V}$, this implies the Mather's theorem, in the compact case. 
\subsection{Thanks}
I thank J. Milnor for his interest on this project. I am grateful to A. Rovella for the communication of his results.  I am  very thankful to J.N. Mather for fruitful conversations and for providing me the unpublished preprint of N. Baas on stability of composed mappings. This work have been done in the Institute for Mathematical Science, in State University of New York at Stony Brook, whose hospitality is gratefully acknowledged.

\section{Proof of the main result} 
The proof is split into two parts. The first is the construction of a geometry which is preserved by the dynamics and which is persistent for small perturbations. For this end, we will use the formalism of laminations and some equivalent of the Hirsch-Pugh-Shub theory. This part is mostly dynamic and geometric. 
The second part consists of showing the structural stability of the dynamics with respects to this geometry. More precisely we will show a generalization of the Mather's theorem on equivalent stability of composing mappings in the larger context of laminations and by allowing these maps to overlap. This part uses mostly ingredients from singularity theory.  
As both parts use the formalism of laminations, we shall first define them rigorously.
\subsection{Definition of the laminations and their morphisms}
\subsubsection{Definition of lamination}
Let us consider a locally compact and second-countable metric space $L$ covered by open subsets $(U_i)_i$ called
\emph{distinguished open subset}, endowed with homeomorphisms $h_i$ from $U_i$ onto $V_i\times T_i$, where $V_i$ is an open set of $\mathbb R^d$ and $T_i$ is a metric space.
The \emph{charts} $(U_i, h_i)_i$ define a ($C^\infty$) \emph{atlas} of a lamination structure on $L$ if the \emph{coordinate changes} $h_{ij}=h_j\circ h_i^{-1}$ can be written in the form:
\[h_{ij}(x,t)=(\phi_{ij}(x,t),\psi_{ij}(x,t)),\]
where $\phi_{ij}$ takes its values in $\mathbb R^d$, $\psi_{ij}(\cdot,t)$ is locally constant for any $t$, and the partial derivatives $(\partial_x^s \phi_{ij})_{s=1}^\infty$ exist and are continuous on \emph{all} the domain of $\phi_{ij}$.
A \emph{lamination} is such a metric space $L$ endowed with a maximal $C^\infty$-atlas $\mathcal{L} $.
A \emph{ plaque } is a subset of $L$ of the form $h_i^{-1}(V_i^0\times\{t\})$, for a
chart $h_i$, $t\in T_i$ and a connected component $V_i^0$ of $V_i$. 

%A plaque that contains a point $x\in L$ will be denoted by $\mathcal{L}_x$; the union of the plaques containing $x$ and of diameter less than $\epsilon>0$ will be denoted by $\mathcal L_x^\epsilon$. As the diameter is given by the metric of $L$, the set $\mathcal L_x^\epsilon$ is -- in general -- not homeomorphic to a manifold. 

The \emph{leaves } of $\mathcal L$ are the smallest subsets of $L$ which contain any plaques that intersect them.
If $V$ is an open subset of $L$, the set of the charts $(U,\phi)\in\mathcal{L}$ such that $U$ is included in $V$, forms a lamination structure on $V$, which is denoted by $\mathcal{L}_{|V}$.

The reader might read \cite{MR1760843} and \cite{berlam} for example of laminations.

%\begin{exem}
%\begin{itemize}
%\item[-] A manifold is a lamination whose leaves are connected component of the manifold.
%\item[-] A $C^r$-foliation on a connected manifold induces a $C^r$-lamination structure.
%\item[-] A locally compact and second-countable metric space defines a lamination of dimension zero.
%\item[-] If $K$ is a locally compact subset of $\mathbb S^1$, then the manifold structure of the circle $\mathbb S^1$ induces on $\mathbb S^1\times K$ a $C^\infty$-lamination structure whose leaves are $\mathbb S^1\times \{k\}$, for $k\in K$.
% \item[-] Let $M$ be the cylinder $\mathbb S^1\times \mathbb R$. 
% Let $\pi$ be the canonical projection of $\mathbb R$ onto $\mathbb S^1\cong \mathbb R/\mathbb Z$. 
% \[\mathrm{Let}\; L:= \Big\{({\theta},y)\in M: \; y= \arctan(\overline{\theta}), \; \pi(\overline{\theta})=\theta\Big\}\cup \mathbb S^1\times \{-\pi/2,\pi/2\}.\] 
% The compact space $L$ is canonically endowed with a $1$-dimensional lamination structure which consists of the leaves $\mathbb S^1\times \{-\pi/2\}$, $\mathbb S^1\times \{\pi/2\}$, and a last one which spirals over these two circles.
% \item[-] The stable foliation of an Anosov $C^r$-diffeomorphism induces a $C^r$-lamination structure whose leaves are the stable manifolds. \end{itemize}
%\end{exem}
\subsubsection{Morphisms of laminations}
For $r\in[\![0, \infty]\!]$, a \emph{$C^r$-morphism (of laminations)} from $(L,\mathcal{L})$ to $(L',\mathcal{L}')$ is a continuous map $f$ from $L$ to $L'$ such that, seen via charts $h$ and $h'$, it can be written in the form:
\[h'\circ f\circ h^{-1} (x,t)= (\phi (x,t),\psi(x,t)),\]
where  $\psi(\cdot , t)$ is locally constant, $\phi$ takes its values in $\mathbb R^{d'}$ and has its $r$-first derivatives with respect to $x$ that are continuous on \emph{all} the domain of $\phi$. When $r=\infty$, the morphism is \emph{smooth}.
If, moreover, the linear map $\partial_x \phi(x,t)$ is always one-to-one (resp. onto), we will say that $f$ is an \emph{immersion (of laminations)} (resp. \emph{submersion}).
An \emph{ isomorphism (of laminations)} is a bijective morphism of laminations whose inverse is also a morphism of laminations.
An \emph{embedding (of laminations)} is an immersion which is a homeomorphism onto its image.
An \emph{endomorphism of $(L,\mathcal{L})$} is a morphism from $(L,\mathcal{L})$ into itself.
% We denote by:
% \begin{itemize}
%\item[-] $Mor^r(\mathcal{L},\mathcal{L}')$ the set of the $C^r$-morphisms from $\mathcal{L}$ into $\mathcal{L}'$,
%\item[-] $Im^r(\mathcal{L},\mathcal{L}')$ the set of the $C^r$-immersions from $\mathcal{L}$ into $\mathcal{L}'$,
%\item[-] $Su^r(\mathcal{L},\mathcal{L}')$ the set of the $C^r$-submersions from $\mathcal{L}$ into $\mathcal{L}'$,
%\item[-] $Iso^r(\mathcal{L},\mathcal{L}')$ the set of the $C^r$-isomorphisms from $\mathcal{L}$ onto $\mathcal{L}'$,
%\item[-] $Emb^r(\mathcal{L},\mathcal{L}')$ the set of the $C^r$-embeddings from $\mathcal{L}$ into $\mathcal{L}'$,
%\item[-] $End^r(\mathcal{L})$ the set of the $C^r$-endomorphisms of $\mathcal L$.\end{itemize}
We denote by $T \mathcal{L}$ the vector bundle over $L$ whose fiber at $x\in L$, denoted by $T_x \mathcal{L}$, is the tangent space at $x$ to its leaf. If $f$ is morphism from $\mathcal{L}$ to $\mathcal{L}'$, we denote by $Tf$ the bundle morphism from
$T\mathcal{L}$ to $T \mathcal{L}'$ over $f$ induced by the differential of $f$ along the leaves of $\mathcal L$.

\subsubsection{Equivalent Classes of morphisms and their topologies}\label{Top}
We will say that two morphisms $f$ and $f'$ from a lamination $(L,\mathcal{L})$ into a lamination $(L',\mathcal{L}')$ are \emph{equivalent} if they send each leaf of $\mathcal L$ into the same leaf of $\mathcal L'$. 
We will deal with two topologies on such an equivalent class. Let us describe a base of neighborhoods of some morphism $f$ from $\mathcal L$ to $\mathcal L'$ in each of these topologies. For this end, let us fix a cover $(K_i)_{i\in \mathbb N}$ of $L$ by compact subsets s.t.:
\begin{itemize}
\item[-] the cover $(K_i)_i$ is locally finite: any point of $L$ has a neighborhood which intersects finitely many subsets of this family,
\item[-] for every $i$, the compact subsets $K_i$ and $f(K_i)$ are included in distinguished open subsets endowed with charts $(h_i,U_i)$ and $(h'_i,U'_i)$ respectively. 
\end{itemize}
We define $(\phi_i,\psi_i)$ by $ h'_i\circ f\circ h_i^{-1}= (\phi_i ,\psi_i)$ on $h_i(K_i)$.
\begin{defi} In the \emph{Whitney $C^r$-topology}, with $r<\infty$, a base of neighborhoods of $f$ in its equivalent class  $Mor_f^r(\mathcal L,\mathcal L')$ is given by the following open subsets with $(\epsilon_i)_i$ going all over the set of families of positive real numbers:

\[\Omega:=\Big\{f'\in Mor^r_f (\mathcal{L} ,\mathcal{L}')\;:\; \forall i\ge 0\;
f'(K_i)\subset U'_i \; \mathrm{and \; with } \; \phi'_i\;\mathrm{defined\; by}\; \quad\]
\[h'_i\circ f'\circ h_i^{-1}=(\phi'_i ,\psi_i),\;\mathrm{we\; have }\; \max_{h_i(K_i)}\Big(\sum_{s=1}^r\|\partial_x^s\phi_i-\partial_x^s\phi'_i\|\Big)<\epsilon_i\Big\}.\]

The \emph{Whitney $C^\infty$-topology} is the union of all the Whitney $C^r$-topologies, for $r\ge 1$.
The Whitney $C^r$-topology is denoted by $\mathcal W^r$.
\end{defi}
\begin{defi} In the \emph{compact-open $C^r$-topology} with $r<\infty$, a base of neighborhoods of $f$ in its equivalent class is given by the following open subsets with $\epsilon>0$ and $i\ge 0$:

\[\Omega:=\Big\{f'\in Mor^r_f (\mathcal{L} ,\mathcal{L}')\;:\; 
f'(K_i)\subset U'_i \; \mathrm{and \; with} \; \phi'_i\;\mathrm{defined\; by }\; \quad\]
\[h'_i\circ f'\circ h_i^{-1}=(\phi'_i ,\psi_i),\; \mathrm{we \; have}\; \max_{h_i(K_i)}\Big(\sum_{s=1}^r\|\partial_x^s\phi_i-\partial_x^s\phi'_i\|\Big)<\epsilon\Big\}.\]
The \emph{$C^\infty$-compact-open topology} is the union of all the $C^r$-compact-open topologies, for $r\ge 1$.
The compact-open $C^r$-topology is denoted by $CO^r$.
\end{defi}
We remark that when $L$ is compact, both topologies are the same.

We notice that when the laminations $(L,\mathcal L)$ and $(L',\mathcal L')$ are manifolds, both definitions are consistent with the usual ones.

\subsection{Results of the geometric part}
The following theorem describes the geometry that the dynamics preserves:
\begin{thm}\label{dyna1}
If $f$ is a $C^\infty$-endomorphism of a smooth manifold $M$ satisfying the hypotheses of the main result, then the stable manifolds of the points of the attractor $A$ form the leaves of a $C^\infty$-lamination $(L^s,\mathcal L^s)$. 

Moreover the compact set $\hat R:=cl\big(\cup_{n\ge 0} f^{-n}(R)\big)$ is canonically endowed with a lamination structure $\mathcal R$ whose leaves are backward images by $f$ of points of $R$. Moreover the lamination $(\hat R, \mathcal R)$ is $r$-normally expanded, for any $r\ge 1$. This means that there exist $C>0$ and $\lambda>1$ s.t. 
for every $x\in R$, for every $u\in T_x\mathcal R$, $v\in T_x\mathcal R^\bot$, $n\ge 0$:
\[\|p\circ T_xf^n(v)\|\ge C\cdot\lambda^n \max(1, \|T_xf^n(u)\|^r),\]

with $p$ the orthogonal projection of $TM_{|\hat R}$ onto $T\mathcal R^\bot$.
%repulsive and equal to the disjoint union of two compact subsets $(R_0, R_1)$ satisfying:
%\begin{itemize}
%\item The closure of $L$ is equal to $L\cup R_1$,
%\item $R_0$ is a union of connected components of $M$.
%\end{itemize}
Also $M$ is equal to the disjoint union of $L^s$ with $\hat R$.
\end{thm}
We will prove this theorem in Section \ref{geodyn}.

The following theorem states the persistence of this geometry:
\begin{thm}\label{basic} Let $f$ be a smooth endomorphism satisfying the hypotheses of the main result. Let $(R,\mathcal R)$ and $(L^s,\mathcal L^s)$ be the laminations provided by Theorem \ref{dyna1}. Let $L'^s$ be a precompact, open subset of $L^s$ whose closure is sent into $L^s$ by $f$.

Then, for any endomorphism $f'$  $C^\infty$-close to $f$, there exists two smooth embeddings $i_R$ and $i_s$ of respectively $(\hat R,\mathcal R)$ and $(L'^s,\mathcal L^s_{|L'^s})$ into $M$ and two smooth endomorphisms $f'_R$ and $f'_s$ of these laminations such that:

\begin{itemize}
\item $f'_R$ and $f'_s$ are equivalent and $CO^\infty$-close to the restrictions of $f$ to respectively $\hat R$ and $L'^s$.
\item $i_R$ and $i_s$ are $CO^\infty$-close to the canonical inclusions of respectively  $(\hat R,\mathcal R)$ and $(L'^s,\mathcal L^s_{|L'^s})$ into $M$,
\item the following diagrams commute:
\[\begin{array}{rcccl}
& &f' & &\\ 
&M & \rightarrow&M &\\
i_R&\uparrow& &\uparrow&i_R\\
&\hat R & \rightarrow&\hat R &\\
& &f'_R & & \end{array}\qquad 
\begin{array}{rcccl}
& &f' & &\\ 
&M & \rightarrow&M &\\
i_s&\uparrow& &\uparrow&i_s\\
&L'^s & \rightarrow&L'^s &\\
& &f'_s & & \end{array}.\]
\end{itemize}\end{thm}

This theorem will be shown in Section \ref{geodyn}. 
\begin{rema} Even if we do not need this to show the main theorem, by using \cite{Bermem}, we can easily show that the embeddings $i_R$ and $i_S$ are the restrictions of a homeomorphism of $M$ which respects smoothly all the laminations.\end{rema} 
\subsection{Result of the singularity theory part}

The main remaining difficulty of the proof of the main theorem is to conjugate $f'_s $ to $f_{|L'^s}$. Such a conjugation follows basically from techniques of singularity theory, and require the following definition to be stated.

\begin{defi}
A morphism $f$ from a lamination $(L',\mathcal L')$ into $(L,\mathcal L)$ is \emph{transversally bijective} if for any point $x\in L'$, there exist charts $(U',\phi')$ and $(U,\phi)$ of respectively $x$ and $f(x)$ s.t. :\begin{itemize}\item $f$ sends $U$ into $U'$,
\item $\phi^{-1}\circ f\circ \phi'$ has its transverse component which is bijective.
\end{itemize}\end{defi}

The following theorem generalizes the one of Mather on equivalent stability of composed mapping of compact manifolds. 
\begin{thm}\label{fonda}
Let $(L,\mathcal L)$ be a $C^\infty$-lamination.  Let $C$ be a compact subset of $L$.
Let $f$ be a smooth endomorphism of $(L,\mathcal L)$ satisfying: 
\begin{itemize}
\item $f$ is transversally bijective and proper,
\item there exists $N>0$ such that for every $x\in C$, an iterate $f^n(x)$ does not belong to $C$, for some $n\le N$,
\item the following map is surjective:
\[\chi^\infty(\mathcal L)
\rightarrow \chi^\infty(f)
\]
\[\sigma\mapsto \sigma\circ f-Tf\circ \sigma,\]
with $\chi^\infty(\mathcal L)$ the space of $C^\infty$-sections of $T\mathcal L$ and $\chi^\infty(f^*)$ the space of $C^\infty$-sections of $f^*T\mathcal L\rightarrow L$.
\end{itemize} 
Then for every $f'$ close to $f$ for the compact-open $C^\infty$-topology, there exists an isomorphism $I(f')$ of $(L,\mathcal L)$ such that for every $x\in C$: 
\[f'\circ I(f')(x)=I(f')\circ f(x)\]
\end{thm}
This theorem will be proved in the last section of this work.
\subsection{Proof of the main theorem}
\begin{proof}[Proof of Theorem \ref{main}]

Le $f$ be a smooth endomorphism of a compact manifold $M$, satisfying the hypotheses of the main theorem. Let $(\hat R, \mathcal R)$ and $(L^s, \mathcal L^s)$ be the laminations provided by Theorem \ref{dyna1}. Let $f'$ be $C^\infty$-close to $f$ and let $i_R$, $i_s$  and $f'_s$ be the maps provided by Theorem \ref{basic}. 
We remind that $f'_s$  is close the restriction of $f$ to the lamination $(L^s ,\mathcal L^s )$ for the compact-open $C^\infty$-topology. 

We put $L:= L^s\setminus \cup_{n\ge 0} f^{-n}(A)$.

We will show in Section \ref{the section} the following lemma:
\begin{lem} The morphism $f$ %restricted to any open precompact subset $L'$ of $L$ 
satisfies the hypotheses of Theorem \ref{fonda}. \end{lem}
The difficulty of this lemma is to pass from infinitesimal stability for manifold to infinitesimal stability for lamination.
This lemma together with Theorem \ref{fonda} provide the structural stability of $f$ on any compact subset of $L$.
Let us chose carefully this compact subset.
As $A$ is an attractor on which $f$ acts bijectively, there exists a small open neighborhood $U$ of $A$ such that:
\begin{itemize}
\item
the restriction of $f$ to $U$ is a diffeomorphism,
\item $f$ sends $cl(U)$ into $U$,
\item $A$ is the maximal invariant subset of $U$: $A=\cap_{n\ge 0} f^{n} (U)$.
\end{itemize}

Let $U^*:= \cup_{n> 0} f^{-n}_{|U^c}(f(U))$. For $U$ sufficiently small, the open subsets $U^*$ and $U$ have their closures disjoint from the singularities $S$ and the critical values $f(S)$ of $f$.

Let $N>0$ be such that $f^{-N}(U)$ contains $S$. Let $K$ be the compact subset:
\[K:= cl\Big(f^{-N}(U)\setminus (U^*\cup U)\Big).\]
We notice that $K$ satisfies the following conditions:
\begin{enumerate}\item[(1)]
$K$ is a neighborhood of the singularities and their image included in $L$,
\item[(2)] for every $x\in K$, if some iterate $f^k(x)$ belongs to $K$ then all the iterates $(f^j(x))_{j=0}^k$ belong to $K$.\end{enumerate}

In order to apply Theorem \ref{fonda}, let us associate a perturbation $f^\#$ of $f_{|L}$ such that the restrictions $f^\#_{|K}$ and $f'_{s|K}$ are equal. For we take a function $\rho\in Mor^\infty(\mathcal L, \mathbb R)$ with compact support and equal to 1 on $K$. Then given an isomorphism $\phi$ from $T\mathcal L$ to a neighborhood of the diagonal of $\mathcal L\times \mathcal L$ which sends a neighborhood of the zero section to the diagonal, we put:
\[f^\#:\; x\in L\mapsto \phi\big( \rho(x)\phi^{-1}(f(x), f'(x)\big).\]

We can now apply Theorem \ref{fonda} with $C:=K$ and for $f_{|L}$. Let $I(f'^\#)$ be the provided isomorphism of $(L,\mathcal L)$ associated to $f'^\#$.

 We define now:
\[h_0:\; x\in \hat K\cup \hat R\mapsto  \left\{\begin{array}{cl} i_R(x)& \mathrm{if}\; x\in \hat R\\
i_s\circ I(f'^\#)(x) & \mathrm{if} \; x\in K\end{array}\right.\]
We notice that $h_0$ respects the lamination $\mathcal L_s$ and that for $x\in K\cup  \hat R$ we have:
\[ f'\circ h_0(x)=h_0\circ f(x).\]
  
Let $M_0$ be the union of the connected components of $M$ which intersect the non-wandering set $\Omega$ of $f$. We denote by $K_0$ the intersection $K\cap M_0$, by $D^n:=f^n(U)\setminus f^{n+1}(U)$ for $n>0$ and by $D^n:= f^{n}_{|M_0}(K_0)\setminus f^{n+1}_{|M_0}(K_0) $ for any $n<0$. We notice that $A\cup  \bigcup_{n>0} D^n$ is a neighborhood of $A$. Let $R_0:= M_0\cap \hat R$. Let us show that $M_0$ is contained in the disjoint union   
\[ A\cup  K\cup  R_0\cup  U^*\cup  \bigcup_{n\not = 0} D^n .\]
Let $x\in M_0$. As the limit set of $x$ is included in $\Omega=A\cup R_0$, every large iterate $f^n(x)$ is either close to $A$ or close to $R_0$.

If $f^n(x)$ is close to $A$, then $x$ is included in  $A\cup  \cup_{n\not = 0} D^n \cup  U^*$. If the orbit of $x$ is never close to  $A$, then $(f^n(x))_n$ belongs to a small neighborhood   of $R_0$, for $n$ sufficiently large. As $R_0$ is repulsive and $f_{|M_0}$-invariant ($f^{-1}(R_0)\cap M_0=R_0$), $R_0$ is locally maximal: 
\[R_0=\cap_{n\ge 0} f^{-n} (V), \quad \mathrm{for\; some\; neighborhood}\; V \mathrm{of}\; R_0.\]
  
This implies that $f^n(x)$ belongs to $R_0$. This completes the proof of the following inclusion:
\begin{equation}\label{equalitynb} M_0\subset A \cup  R_0 \cup  U^*\cup K \cup  \bigcup_{n\not= 0} D^n.\end{equation}
We endow $M$ with an adapted metric to the hyperbolicity of $R_0$ and $A$. 
By restricting $U$ and taking $N$ sufficiently large, we can suppose that $f'$ is uniformly expending on $\cup_{n<0} D^n\cup R_0$ and uniformly contracting along the leaves of $\mathcal L_{|U}^s$. Thus we can suppose  
$\epsilon>0$ small enough and then $f'$ close enough to $f$ such that:
\begin{itemize}
\item for every $x\in cl(\cup_{n<0} D^n \cup R^0)$,   the restriction of $f'$ to any  ball $B(x,\epsilon)$ centered at $x$ and with radius $\epsilon$  is an expanding diffeomorphism onto a neighborhood of $B(f(x),\epsilon)$,
\item For every $x\in U$, the restriction of $f'$ to $\mathcal L_x'$ is contracting with precompact image in $\mathcal L'_{f(x)}$, where $\mathcal L'_x$ and $\mathcal L'_{f(x)}$  are the image by $i_s$ of respectively the union of $\epsilon$-$\mathcal L^s$-plaque containing $x$ and $f(x)$ respectively.
\end{itemize}
For $x\in A$, let $h_A(x)$ be the unique element of the intersection $\cap_{n\ge 0} f'^n\big(\mathcal L'_{f_{|U}^{-n}(x)}\big)$.
  
\[ \mathrm{Let}\; h_1:= x\in M_0\setminus U^* \mapsto \left\{ \begin{array} {cl} 
h_0(x) & \mathrm{if} \; x\in R_0\cup K\\
f'^n\circ h_0\circ f^{-n}_{|U} & \mathrm{if} \; x\in D^n, \; n>0\\
f'^{-1}_{|B(x,\epsilon)}\circ \cdots \circ f'^{-1}_{|B(f^{n-1}(x),\epsilon)}\circ h_0 \circ f^n(x) & \mathrm{if }\; x\in D^{-n}, n>0\\
h_A(x)                                                                                            & \mathrm{if }\;  x\in A\end{array}\right.\]

We notice that $h_1$ is well defined since the union (\ref{equalitynb}) is disjoint. Also since $f'$ and $f$ are conjugated via $h_0$ on $K$, the map $h_2$ is continuous on $M_0\setminus (A\cup  R_0\cup U^*)$.

As $h_0$ is close to the identity, by commutativity of the diagram and expansion of $R_0$, for every $x\in R_0$, the intersection:
\[ \cap_{n>0} f'^{-1}_{|B(x,\epsilon)}\circ \cdots \circ f'^{-1}_{|B(f^{n-1}(x),\epsilon)}\big(B(f^{n}(x),\epsilon\big)\]
is exponentially decreasing to $\{h_0(x)\}$. Consequently, for every $x'\notin R_0\cup U^*$ close to $x$ and so in $D^{-n}$ with $n$ large, $h_1(x)$ belongs to $f'^{-1}_{|B(x,\epsilon)}\circ \cdots \circ f'^{-1}_{|B(f^{n-1}(x),\epsilon)}\big(B(f^{n}(x),\epsilon)\big)$ and so is close to $h_0(x)$. Thus $h$ is continuous at $R_0$.

For $x\in A$, we recall that the intersection $\cap_{n\ge 0} f'^n(\mathcal L'^\epsilon_{f^{-n}_{|U}(x)})$
is exponentially decreasing to $\{h_1(x)\}$.
Also for every $x'$ close to $x$ and so in   $D^n$ with $n$ large,  the point $h_1(x')$ belongs to the intersection 
$\cap_{k=0}^{n'} f'^k (\mathcal L'^\epsilon_{f^{-k}_{|U}(x')})$
whose diameter is small for $n'<n$ large. As $(\mathcal L'_{|f^{-k}_{|U}(x')})_{k=0}^{n'}$ and $(\mathcal L'_{|f^{-k}_{|U}(x)})_{k=0}^{n'}$ are close, $h_1(x')$ and $h_1(x)$ are close. This proves that $h_1$ is continuous at $A$. 

Since $U^*$ has its closure disjoint from the singularities of $f$, $f$ is a local diffeomorphism on a neighborhood of $cl(U^*)\cap M_0$. Consequently, for $\epsilon$-small enough and then $f'$ close enough to $f'$, the restriction $f'_{|B(x,\epsilon)}$ is a diffeomorphism onto its image, for every $x\in U^*$. As the connected components of $U^*$ accumulate on $R_0$, we can define similarly for $f'$ close enough to $f$:
\[h_2:\; x\in M_0\mapsto  \left\{\begin{array}{cl} 
h_1(x)& \mathrm{if} \; x\in M_0\setminus U^*\\
f'^{-1}_{|B(x,\epsilon)}\circ \cdots \circ f'^{-1}_{|B(f^{k-1}(x),\epsilon)}\circ h_1\circ f^k(x) & \mathrm{if} \; x\in f^{-k}_{|U^c} (U) \setminus f^{-k+1}(U)\end{array}\right.\] 

The continuity of $h_2$ on $M_0\setminus R_0$ follows from the conjugacy of $f'$ and $f$ via $h_1$ on $M_0\setminus U^*$. The continuity of $h_2$ at $R_0$ is proved similarly as we did for $h_1$.

By expansiveness of $f_{|A\cup R_0}$, $h_1$ is injective on $A\cup R^0$. Also by definition, $h_1$is injective on $M^0\setminus (A\cup R_0\cup U^*)$. By using the attraction-repulsion of $A-R_0$, we get that $h_1$ is injective on $M\setminus U^*$. By local inversion of $f_{|U^*}$, we have the injectivity of $h_2$.

Therefore $h_2$ is a continuous injective map from $M_0$ into itself, $C^0$-close to the identity. By compactness of $M_0$ this implies that $h_0$ is a homeomorphism of $M_0$.

Let us finally construct $h$. Let $K'$ be a neighborhood of the singularities, included in the interior of $K$ and  satisfying properties $(1)$ and $(2)$ of $K$.  
Let $M_s:=M\setminus (M_0\cup K')$. Since the restriction of $f$ to the $M_s$ is a submersion, we can foliate $M_s$ by the two following ways. For $x\in M_s$, let $m_x>0$ be minimal such that $y:=f^{m_x}(x)$ belongs to $M_0$. Let $F_y$ be the submanifold equal to 
$\cup_{k>0} f^{-k}(\{y\})\setminus (M_0\cup K') $. Let $F'_y$ be the submanifold equal to 
$\cup_{k>0} f'^{-k}(\{y\})\setminus (M_0\cup K') $. 
We notice that $(F_y)_{y}$ and $(F'_y)_{y}$ are the leaves of two foliations on $M_s$. We notice that the restriction of the exponential map associated to the metric of $M$:
\[\phi:\; TF'^\bot_y \rightarrow M\]
\[(x,u)\mapsto \exp_x(u)\]
is a diffeomorphism from a neighborhood of the zero section of $TF'^\bot_y\rightarrow F'_y$ onto a neighborhood $V$ of $F'_y$.
Let $\pi_y:\; V\rightarrow F'_y$ be the composition of $\phi^{-1}$ with the projection $TF'^\bot_y \rightarrow F'_y$. Let $K''$ be a compact neighborhood of $K'$ in $int(K)$. We notice that the following map is well defined for $f'$ close enough to $f$.  
\[h_{2}':\;x\in M\setminus (M_0\cup K'')\mapsto \pi_{h_2(y)} (x),\quad \mathrm{if}\; x\in F_{y}.\]
We notice that $h_0$ and $h_2'$ sends every point $x\in K\setminus K''$ into $F_{h_2(y)}'$.

Thus we may patch $h'_2$ and $h_2$ to a map $h: \; M\rightarrow M$ satisfying that:
\begin{itemize}
\item $h$ is equal to $h_2$ on $M_0$, to $h_0$ on $K''$ and to $h_2'$ on $M\setminus (K\cup M_0)$,
\item $h$ sends any points $x\in M\setminus K'' $ to a points of $F'_{h_2(y)}$, with $y:= f^{m_x}(x)$,
\item The restriction of $h$ to $L$ is a smooth embedding of $\mathcal L$ into $M$,
\item $h$ is a homeomorphism onto its image $C^0$-close to the canonical inclusion.\end{itemize}

We notice that $h$ is a homeomorphism of $M$ close to the identity. Moreover, it satisfies:
\[f'\circ h=h\circ f.\]
\end{proof}
\section{Partition of the manifold by invariant and persistent laminations}\label{geodyn}
\subsection{Construction of the invariant laminations}

In this subsection we prove  Theorem \ref{dyna1} which states the existence of a splitting of $M$ into two laminations $(L^s,\mathcal L^s)$ and $(\hat R,\mathcal R)$.
\subsubsection{The disjoint union of $L$ with $\hat R$ is $M$}
By the work of Przytycki \cite{Przy}, there exists a neighborhood $V$ of $R$ s.t. $R$ is the maximal invariant of $V$:
\[R=\bigcap_{n\ge 0} f^{-n}(V).\]

As the limit set is included in the non-wandering set, the last equality implies that the complement of the basin of $A$ is $\cup_{n\ge 0} f^{-n}(R)$. Thus the last union is closed and so equal to $\hat R$. To summary we have:
\[ \hat R= \bigcup_{n\ge 0} f^{-n}(R)\quad \mathrm{and }\quad M= \hat R\cup L.\]
\subsubsection{ Construction of $\mathcal R$}
As the forward orbit of the critical set does not intersect $R$, the backward images of point of $R$ form a partition of $\hat R$ by compact submanifolds. Moreover these submanifolds depend continuously on the point, and so are the leaves of a lamination $\mathcal R$ on $\hat R$.

Let $M_0$ be the union of the connected components of $M$ which intersect $\Omega$. We notice that the restriction $\mathcal R_{|M_0\cap \hat R}$ is a $0$-dimensional lamination: its leaves are points. Also since $M$ is compact, there exists $N\ge 0$ s.t. $f^{-N}(M_0)=M$. Consequently to prove that $(\hat R, \mathcal R)$ is a normally expanded lamination, we only need to prove that $\hat R\cap M'$ is repulsive; this is proved in \cite{Rov1} Lemma 1.
   
  \subsubsection{Construction of $\mathcal L^s$}
The existence of a laminar structure $\mathcal L^s$ on $L^s$ is a consequence of the following proposition:
\begin{prop} \label{lamhyper}
Let $r\in \{1,\dots, \infty\}$, $f$ a $C^r$-endomorphism of a manifold $M$ and $K$ a hyperbolic compact subset of $M$.
We suppose that the singularities of $f$ have their orbits disjoint from $K$ and that for every $y\in K$, for every $n\ge 0$, the map $f^n$ is transverse to a local stable manifold of $y$. 

 Then the union $W^s(K)$ of stable manifolds of points of $K$ is the image of a $C^r$-lamination $(L,\mathcal L)$ immersed injectively.

Moreover if every stable manifold does not accumulate on $K$, then $(L,\mathcal L)$ is a $C^r$-embedded lamination.
 \end{prop}

\begin{proof} We endow $M$ with an adapted metric $d_M$ to the hyperbolic compact subset $K$. For a small $\epsilon >0$, the \emph{local stable manifold of diameter $\epsilon$ of $x\in K$} is the set of points whose (forward) orbit is $\epsilon$-distant to the orbit of $x$. Let
$W^s_\epsilon(K)$ be the union of stable manifolds of points in $K$ of diameter $\epsilon$. For $\epsilon$ small enough, the closure of $W^s_\epsilon(K)$ is sent by $f$ into $W^s_\epsilon(K)$ and supports a canonical $C^r$-lamination structure $\mathcal L_0$. Let $C$ be the subset $W^s_\epsilon(K)\setminus f^2\big(cl(W^s_\epsilon(K))\big)$.
For $i>0$, we denote by $C_i$ the set $f^{-i}(C)$ and by $C_0$ the set $W^s_\epsilon(K)$.
Consequently, the union $\cup_{n\ge 0} C_n $ is equal to $W^s(K)$. Moreover, for $k,l\ge 0$, if $C_k$ intersects $C_l$ then $|k-l|\le 1$.

 Let us now construct a metric on $W^s(K)$ such that $( C_n)_n$ is an open cover and such that the topology induced by this metric on $ C_n$ is the same as the one of $M$. For $(x,y)\in  W^s(K)^2$, we denote by $d(x,y)$:
 \[\inf\Big\{\sum_{i=1}^{n-1} d_M(x_i,x_{i+1}); \;n>0,\;
 (x_i)_i\in W^s(K)^n,\;\mathrm{such\; that}\]
\[  x_1=x,\; x_n=y,\; \forall i\exists j: (x_i,x_{i+1})\in C_j^2\Big\}.\]
We remark that $d$ is a distance with the announced properties. Let $L$ be the set $W^s(K)$ endowed with this distance. We notice that if every stable manifold does not accumulate on $K$, then the topology on $L$ induced by this metric and the metric of $M$ are the same. In other words $L$ is embedded.

 For $i>0$, we now construct on the open subset $C_i$ a $C^r$-lamination structure $\mathcal L_{i}$.
Let $(U_k,h_k)_k$ be an atlas of $\mathcal L_{0|C}$ of the form:
\[h_k:\; U_k\rightarrow \mathbb R^{d_k}\times T_k\]
\[x\mapsto \big(\phi_k(x),\psi_k(x)\big) \]
Let $U_k' :=f^{-i}(U_k)$ and 
\[\psi_k' :\;U_k' \rightarrow  T_k\]
\[x\mapsto \psi_k\circ f^{i}(x) \]
By shrinking a slice $U_k$ (and hence $U_k'$)  and by using the transversality of $f$, there exists a neighborhood  $T_k' $ of any $t\in T_k$, such that $(\psi_k'^{-1}(t'))_{t'\in T_k' }$ is a family of manifolds that are all diffeormorphic to $M_t:=\psi_k'^{-1}(t)$ by a diffeomophism that depends $C^r$-continuously on $t'\in T_k' $.
Let us denote by $\phi_k' : U_k''\rightarrow M_t$ this continuous family of $C^r$ diffeormorphisms,  with $U''_k:= \Psi_k'^{-1}(T_k')$. Now let $(U_\alpha, h_\alpha)_{\alpha\in A}$ be a $C^r$-atlas of the manifold $M_t$. For each $\alpha$, let  $U_k^\alpha$ be the open subset of $U_k''$ equal to $\phi_k'^{-1}(U_\alpha)$. We notice that the map:
\[h_k^\alpha:\; U_k^\alpha\rightarrow \mathbb R^d_{\alpha}\times T'_k\]
\[x\mapsto ( h_\alpha\circ \phi_k'(x), \psi'_k(x))\]
is a chart of an atlas of lamination $\mathcal L_i$ on $C_i$, for $k\ge 0$, $t\in T_k$, and $\alpha\in A$.

Moreover, for $i,j$ consecutive, the restriction of $\mathcal L_i$ and $\mathcal L_j$ to $C_i\cap C_j$ are equivalent.
 Thus, the structures $(\mathcal L_i)_{i\ge 0}$ span a $C^r$-lamination structure $\mathcal L$ on $L$.
\end{proof}
\subsection{Persistence of the laminations}
In this section we prove Theorem \ref{basic}.

The existence of $i_R$ and $f_R$ follows from the fact that $(\hat R, \mathcal R)$ is a compact, $r$-normally expanded lamination for any $r\ge 1$.
Thus by Theorem 0.1 of \cite{berlam}, the lamination $(\hat R, \mathcal R)$ is $C^r$-persistent. This means that for   $f'$ $C^r$-close to $f$, there exist a $C^r$-embedding $i_R'$ of $(\hat R,\mathcal R)$ into $M$, and a $C^r$-endomorphism $f_R$ of $(R,\mathcal R)$ such that the following diagram commutes:

\[\begin{array}{rcccl}
& &f' & &\\ 
&M & \rightarrow&M_0 &\\
i'_R&\uparrow& &\uparrow&i'_R\\
&\hat R & \rightarrow&\hat R &\\
& &f_R & & \end{array}.\]

Moreover $i'_R$ is $C^r$-close to the canonical inclusion and $f_R$ is $C^r$-close to $f_{|\hat R}$.

By Theorem \ref{dyna1}, the lamination $(\hat R, \mathcal R)$ embedded by $i_R$ is actually of class $C^\infty$. Thus to smooth $i'_R$ and $f_R$ we consider a smooth tubular neighborhood (see Section 1.5 of \cite{berlam}). This is the data of:
\begin{itemize}\item
a smooth laminar structure $\mathcal F$ on $F:= TM_{|\hat R}/ T\mathcal   R$ such that the leaves of $\mathcal F$  are the preimages by $\pi:\; F= TM_{|\hat R}/T\mathcal R\rightarrow L$ of the leaves of $\mathcal L$ and such that $\pi$ is a smooth submersion,
\item  an immersion $I$ from the restriction of $\mathcal F$ to a neighborhood of the zero section $0_F$, such that 
\[I\circ 0_F=id_{\hat R}\]
\end{itemize}

By compactness of $\hat R$, there exists $\epsilon$ such that the restriction of the ball $\mathcal F_x^\epsilon$ centered at $x\in \hat R$ and with radius $\epsilon$ in  the leaf of $0_x$ in $\mathcal F$ is sent diffeomorphically by $I$ onto an open subset of $M$. Let $ F_x^\epsilon$ be the intersection of $\mathcal F_x^\epsilon$ with the fiber $F_x$ of $F\rightarrow \hat R$ at $x$. For $f'$ close to    $f$, the image by $I$ of $F_x^\epsilon$ intersects transversally  at a unique point $i_R(x)$ the image by $i'_R$ of a plaque $\mathcal L_x$ of $x$. We notice that $i_{R|\mathcal L_x}$ is the composition of $i$ with the holonomy from the transverse section $i(\mathcal L_x)$ to $i_{R}(\mathcal L_x)$ along the foliation $(F_{x'}^\epsilon)_{x'\in \mathcal L_x}$. As these submanifolds and foliations are smooth, $i_{R|\mathcal L_x}$ is smooth. As these foliations and manifolds depend continuously on $x$, $i_R$ is a smooth morphism  of $(\hat R,\mathcal R)$ into $M$. As $i_R'(\mathcal L_x)$ is $C^r$-close to $i_R(\mathcal L_x)$, $i_R$ is $C^r$-close to  $i$. In particular $i_R$ is an immersion.

 Also by construction $i_R$ is injective and so, by compactness of $\hat R$, $i_R$ is an embedding. 

Let finally $f'_R:= x\in R\mapsto \pi \circ I^{-1}_{|\mathcal F_{f(x)}^\epsilon} \circ f'\circ i_R(x)$. The composition of a smooth morphisms  $f'_R$ is smooth. Also one easily notes that $f _R$ is equivalent and $C^r$ close to $f_{|R}$. 

The persistence of $(L'^s, \mathcal L^s_{|L'^s})$ is showed similarly, by using this time Theorem 3.1 of \cite{berlam}.

\section{Infinitesimal stability implies stability on the non-wandering set} 
In this section, we prove Theorem \ref{fonda}. Throughout this section we denote by $(L,\mathcal L)$ a $C^\infty$-lamination,  $C$ a compact subset of $L$. Let $f$ be a proper, smooth endomorphism of the lamination $(L,\mathcal L)$ s.t. for some $N\ge 0$, 
$x\in C$ there exists $n\le 0$ s.t. $f^n(x)$ does not belong to $C$.
\subsection{Stability under deformations}
We are going to prove that infinitesimal stability $(\mathcal I)$ implies another property called \emph{stability under deformations}  $(\mathcal D)$.
\subsubsection{Condition $\mathcal D$}
\begin{defi} A \emph{deformation of $f$} is a smooth endomorphism $F$ of the product lamination 
$(L\times \mathbb R, \mathcal L\times \mathbb R)$, of the form:
\[F\;:\; (x,t)\in L\times \mathbb R\mapsto (f_t(x),t)\in L\times \mathbb R,\]
and such that $f_0=f$. % and also $f_t(x)=f(x)$ for every $t\in \mathbb R$ and every $x\notin W$.
We remind that the leaves of the lamination $\mathcal L\times \mathbb R$ are the product of the
leaves of the lamination $\mathcal L$ with $\mathbb R$.\end{defi}

\begin{defi} Let $B$ be a neighborhood of $0\in \mathbb R$. A deformation $F=(f_t)_t$ of $f$ is \emph{trivial relatively to $C\times B$} if there exists a deformation $H$ of the identity of $(L,\mathcal L)$:
\[H:\; (x,t)\in L\times \mathbb R\mapsto (h_t(x),t),\]
such that for all $(x,t)\in C\times B$:
\[h_t\circ f(x)=f_t\circ h_t(x).\]
The automorphism $H$ is a \emph{trivialization} of $F$ (relatively to $C\times B$).
\end{defi}
\begin{defi}[Condition $\mathcal D$] We say that $f$ is \emph{stable under deformations relatively to $C$} if 
for any bounded ball $B$ centered at $0$, for any deformation $F:\; (x,t)\mapsto (f_t(x),t)$ of $f$ $CO^\infty$-close enough to $F_0:= (x,t)\mapsto (f(x),t)$ is trivial relatively to $C\times B$.
\end{defi}
Also the following proposition is obvious:
\begin{prop} If $f$ is stable under $k$-deformations relatively to $C$, then $f$ is \emph{structurally stable relatively to $C$}: for any $f'$ $CO^\infty$-sufficiently close to $f$ there exists an isomorphism $h$ of $(L,\mathcal L)$ such that for any $x\in C$, we have:
\[h\circ f(x)=f'\circ h(x)\]
\end{prop}
%\subsubsection{Condition $\mathcal O$}
%The statement of condition $\mathcal O$ is technical but obviously open.
%
%%First of all, let us notice that there exists $N\ge 0$, such that for any $x\in W$, such that $f^k(x)$ do not belong to $W$ for some $k\le N$.
%For $p\in W$ let $n_p> 0$ be the first integer such that $f^{n_p}(p)$ does not belong to $W$.
%
%Let $s:=$ and $m:=$.
%
%$f$ satisfies condition $\mathcal O$, if for every $p\in W$, there exists a neighborhood $V_p$ of $p$ and a neighborhood $V_f$ of $f$ such that 
%for any $f'\in V_f$ and $S:=\{p_1,p_2,\dots, p_s\}\subset V_p\cap f'^{-n_p}(q)$, with $q:=f^{n_p}(p)$, the following map is onto:
%\[\prod_{y\in\{f'^k(p_i); \; 1\le i\le s \;\mathrm{and}\; 0\le k\le n_p\}} J^m(T\mathcal L)_y \rightarrow 
%\prod_{y\in\{f'^k(p_i); \; 1\le i\le s \;\mathrm{and}\; 0\le k\le n_p-1\}} J^m(f'^*T\mathcal L)_y \]
%\[(\sigma_y)_y\mapsto (\sigma_{f'(y)}\circ f'-Tf\circ \sigma_y)_y,\]
%
%with $J^m(T\mathcal L)_y$ (resp. $J^m(f'^*T\mathcal L)_y$) the space of $m$-jets of sections of $T\mathcal L$ (resp. $f^*T\mathcal L$), that is the space of $C^\infty$-sections of $T\mathcal L$ (resp. $f^*T\mathcal L$) quotiented by the following relation:
%\[\sigma\sim\sigma'\Leftrightarrow \sigma(x)=\sigma(x)+o(d(x,y)^{m+1}),\]
%where $y$ belongs to the leaf of $x$, and $d$ refers to the distance of this leaf as a manifold.
\subsubsection{Sufficiency of the implication $I\rightarrow D$} 
The main remaining difficulty is to prove the following theorem:
\begin{thm}\label{equi} Let $f$ be a proper, $C^\infty$ endomorphism of a lamination $(L,\mathcal L)$. Let $C$ be a compact subset of $L$ such that for some $N>0$ the intersection  $\cap_{n=0}^N f^{-n}(C)$ is empty. If $f$ is transversally bijective and if:
\begin{itemize}
\item[$\mathcal I$)] the following map is surjective:
\[\chi^\infty(\mathcal L)
\rightarrow \chi^\infty(f)\]
\[\sigma \mapsto \sigma\circ f-Tf\circ \sigma,\]
Then: 
\item[$\mathcal D$)]The morphism $f$ is stable under deformations relatively to $C$.%\item[$\mathcal O$)] $f$ satisfies the condition $\mathcal O$.
\end{itemize}\end{thm}
By the previous proposition, this theorem implies Theorem \ref{fonda}.
\subsection{Condition $\mathcal I \Rightarrow$ condition $\mathcal D$}
Let $f$ be infinitesimally stable. 
We want to prove that $f$ satisfies condition $\mathcal D$. Let $\chi^\infty(\mathcal L,\mathbb R)^0$ be the space of smooth vector fields with $\mathbb R$ component equal to 0. Let $\frac{\partial}{\partial t}$ be the canonical unit vector field of the product $\mathcal L\times \mathbb R$ associated to $\mathbb R$.  

\subsubsection{Thom-Levin Theorem}
The following theorem transforms the problem of the existence of trivialization  $H$ to a linear problem. This will allow us to solve this problem algebraically. 
\begin{thm}[Thom-Levine theorem adapted]\label{thom-lev} 
Let $B$ be a subset of $\mathbb R$ and let $W$ be a neighborhood of $C\times B$ in $L\times \mathbb R$.  There exists a $\mathcal W^\infty$-neighborhood $V_\xi$ of $0\in \chi^\infty (\mathcal L\times \mathbb R)^0$. A deformation $F$ of $f$ is trivial relatively to $C\times B$, if there exists  $\xi\in V_\xi$ 
 such that on $W$:
\[\tau_F:=TF\circ \frac{\partial}{\partial t}_{|L\times \mathbb R}-\frac{\partial}{\partial t}\circ F= TF\circ \xi-\xi\circ F.\]
\end{thm}
\begin{rema} Actually the statement of the Thom-Levin theorem is for vector field of manifold and is interested in equivalence and 
not conjugacy as here. Nevertheless the proof of this theorem is an adaptation of the one written by Golubitsky-Guillemin \cite{Gol} P123-127.\end{rema}
\begin{rema}\label{pour 1.7 aussi} The above theorem remains true if $C$ is possibly non-compact but closed in $L$. For such a generalization, the proof bellow works as well.
\end{rema}
Before proving Theorem \ref{thom-lev}, we need a few lemmas.
\begin{lem}\label{3.4}
Let $\xi$ be a $\mathcal W^\infty$-small vector field on $\mathcal L\times \mathbb R$ with zero $\mathbb R$-component. Then there is an automorphism $H$ of $(L\times \mathbb R, \mathcal L\times \mathbb R)$,
which is a deformation of $id_{L}$ satisfying:
\[TH\circ \frac{\partial}{\partial t}\circ H^{-1}=-\xi+\frac{\partial}{\partial t},\]
Moreover, $H$ is $\mathcal W^\infty$-close to the identity.
\end{lem}
\begin{proof} See \cite{Gol} Sublemma 3.4 p125. \end{proof}
\begin{lem}\label{3.5}
By using the same notations as in the above lemma, we have:
\begin{itemize}\item[$(i)$]
$\xi=- T\pi\circ TH \circ \frac{\partial}{\partial t}\circ H^{-1}$,
\item[$(ii)$]
$\xi= TH\circ T\pi\circ TH^{-1} \circ \frac{\partial}{\partial t}$,\end{itemize}
where $\pi\;:\;L\times \mathbb R\rightarrow L$ is the canonical projection. 
\end{lem}
\begin{proof}
The first statement of this lemma is obvious since $T\pi\circ \frac{\partial}{\partial t}=0$ and $T\pi\circ \xi=\xi $. Applying $TH^{-1}$
to both sides of the equation of Lemma \ref{3.4}, we get:
\[TH^{-1}\circ \big(- \xi+\frac{\partial}{\partial t}\big)= \frac{\partial}{\partial t}\circ H^{-1}.\]
Applying $T\pi$ on both sides above, we have:
\[-T\pi\circ T H^{-1}\circ\xi+T\pi\circ T H^{-1} \circ \frac{\partial}{\partial t}=0.\]
As the $\mathbb R$-component of $\xi$ is $0$, so is the $\mathbb R$ component of $TH^{-1}\circ \xi$. Thus:
\[ TH^{-1}\circ\xi=T\pi\circ TH^{-1} \circ \frac{\partial}{\partial t}.\]
By applying $TH$ on both sides of the above equation, we have $(ii)$.
\end{proof}
\begin{proof}[Proof of Theorem \ref{thom-lev}]
It is sufficient to show the existence of a deformation $H$ of the identity such that:
\[F'=H^{-1}\circ F\circ H,\]
is equal to the trivial deformation $F_0$ on $C\times B$. 
We notice that this holds if and only if the following vector field is $0$ on $C\times B$:
\begin{equation}\label{HF'}
\tau_{F'}:= T\pi\circ TF'\circ \frac{\partial}{\partial t}.
\end{equation}

%Let us show the existence of a closed neighborhood $W$ of $C$ in $L'$ such that $f_{|W}$ is still proper. As $L$ is secondly countable and locally compact, there exists an increasing sequence of compact sets $(C_n)_n$  whose union is $C$ and such that $C_{n+1}$ is a neighborhood of $C_n$ for the topology of $C$, for every $n\ge 0$. We chose a compact neighborhood $W_n$ of $C_n$, for each $n$. 
%The family $(W_n)_n$ has to be chosen carefully in order that its union $W$ can be convenient. For this end, we consider an increasing  family of compact subsets $(K_n)_n$ whose union is $L$. For every $n\ge 0$, there exists $N_n\ge n$ such that the preimage by $f_{|C}$ of $K_n$ is included in $C_{N_n}$. We suppose $W_{N_n}$ constructed. For $m\ge N_m$, we can suppose $W_{m}$ sufficiently small in order that the  preimage by $f_{|W_m}$ of $K_n$ can be included in $W_{N_n}$. Then we notice that that the restriction of $f$ to $W:= \cup_n W_n$ is proper.
  
By assumption, there exists a $\mathcal W^\infty$-small vector field $\xi$ on $(L\times \mathbb R,\mathcal L\times \mathbb R)$, whose $\mathbb R$ component is zero and such that:
\[\tau_F =TF\circ \xi-\xi\circ F,\quad \mathrm{on}\; W.\]
Let us construct $H$ and so $F'$ such that $\tau_{F'}$ is zero on $C\times B$.
 By applying Lemmas \ref{3.4} and \ref{3.5}, we get the existence of a deformation: 
\[H:\; L\times \mathbb R\rightarrow L\times \mathbb R\]
so that 
\[\xi=-T\pi\circ TH\circ \frac{\partial}{\partial{t_k}} \circ H^{-1}.\]

As $H$ is $\mathcal W^\infty$-close to the identity, it sends $C\times B$ into $W$.

We recall that $F':= H^{-1}\circ F\circ H$. Let $p$ be in $C\times B$, $q:= H(p)$ and $r:= F\circ H(p)$. Using the fact that $F$ and $H$ are deformations, we have that:
\[TF'\circ \frac{\partial}{\partial t}(p)=\frac{\partial}{\partial {t}}\circ F'(p)+T\pi\circ TH^{-1}\circ \frac{\partial}{\partial t}(r)+
TH^{-1}\Big[ T\pi\circ TF\circ \frac{\partial}{\partial {t}}(q)-TF\circ \xi(q)\Big]\]
One the other hand, by using statement $(ii)$ of Lemma \ref{3.5}, we have:
\[T\pi\circ TH^{-1}\circ \frac{\partial}{\partial {t}}(r)=TH^{-1}\circ \xi(r).\]
The two last equations imply that 
\[TF'\circ \frac{\partial}{\partial t}(p)=\frac{\partial}{\partial {t}}\circ F'(q)+
TH^{-1}\Big[ \xi(r)+T\pi\circ TF\circ \frac{\partial}{\partial {t}}(q)-TF\circ \xi(q)\Big]\]
By assumption:
\[\tau_F=TF\circ \frac{\partial}{\partial {t}}-\frac{\partial}{\partial {t}}\circ F=TF\circ \xi-\xi\circ F, \quad \mathrm{on}\; W\]
We notice that $\tau_F=T\pi\circ TF\circ \frac{\partial}{\partial t}$.
Thus, we have:
\[\tau_{F'}(p):=T\pi\circ TF'\circ \frac{\partial}{\partial t}(p)=
T\pi\circ TH^{-1}\Big[ -\tau_F(q) +T\pi\circ TF\circ \frac{\partial}{\partial {t}}(q)\Big]=0 \]\end{proof}

\subsubsection{Proof of $\mathcal I \Rightarrow\mathcal D$ using the Thom-Levine theorem}
Let $f$, $C$ and $(L,\mathcal L)$ be as stated in the theorem. In particular $f$ is infinitesimally stable. We want to prove that $f$ is stable by deformation, relatively to $C$. Let $B$ be a bounded ball centered at $0$ and $W$ a neighborhood of $C\times B$ in $L$. By the adaptation of the Thom-Levin theorem, it is sufficient to show, for any $CO^\infty$-small deformation $F:\; L\times \mathbb R\rightarrow L\times \mathbb R$ of $f$,
the existence of a $\mathcal W^\infty $-small vector field $\xi\in \chi^\infty(\mathcal L\times \mathbb R)^0$ with $\mathbb R$-component equal to 0 such that on $W$:
\[\tau_F:=TF\circ \frac{\partial}{\partial t}-\frac{\partial}{\partial t}\circ F= TF\circ \xi-\xi\circ F,\]
with $\partial/\partial t\in \chi^\infty(L\times \mathbb R)$ the canonical unit vector field associated to $\mathbb R$.

\paragraph{Local version}
We first prove the existence of $\xi$ locally:
\begin{prop}\label{local:version}
Let $\hat P_1$ be a small plaque of $\mathcal L$ such that $\hat P_{i+1}:= f^{-i}(\hat P_1)$ is disjoint from $\hat P_1$ for every $i> 0$.
Let $P_1$ be a precompact plaque whose closure is included in $\hat P_1$. Let  $P_{i+1}:= f^{-i}(P_1)$, for every $i\ge 0$. 

Let $s> 0$ be such that $\hat  P_{s}$ is non-empty.
Then $(\hat P_i)_{i=1}^s$ consists of manifolds. Moreover, there exists a family of disjoint open neighborhoods $(U_i)_{i=1}^s$ of $(P_i)_{i=1}^s$ s.t. $f$ sends $\Omega':= \coprod_{i=2}^{s} U_i$ into $\Omega:= \coprod_{i=1}^s U_i$, and s.t.  
for any deformation $F$ $CO^\infty$-close to $F_{0}$, 
there exists $\xi\in \chi^\infty (\mathcal L\times \mathbb R)^0$ which is $\mathcal W^\infty$-close to zero with $\mathbb R$-component equal to $0$ and such that :
\[\tau_F(x,t)=TF\circ \xi(x,t)-\xi\circ F(x,t),\quad \forall x\in \Omega',\; t\in \mathbb B.\]\end{prop}
\begin{proof}
The fact that $(\hat P_i)_i$ is a family of manifolds follows from the transverse bijectivity of $f$.

The main interest of the proof is the usefulness of the algebraic tools developed by Mather (\cite{Ma1}, \cite{Ma2}, \cite{Ma3}) and very well written by Tougeron \cite{Tou}.

Let $R_i'$ be the ring $C^\infty(\hat P_i)$ of smooth real functions on $\hat P_i$. 
For $t\in \mathbb R$, let $R_i^t$ be the ring $C^\infty_{\hat P_i\times \{t\}} (\hat P_i\times \mathbb R)$ of smooth germs at $\hat P_i\times \{t\}$ of smooth real functions on $\hat P_i\times \mathbb R$. 
We notice that $R_i'$ is isomorphic to the quotient $R_i^t/I_i^t$, where $I_i^t$ is the ideal of $R_i^t$ formed by the germs equal to $0$ on $\hat P_i\times \{t\}$. 
For $i\in \{1,\cdots , s-1\}$, we notice that the map 
\[\phi_i^t:\; R_{i}^t\rightarrow R_{i+1}^t\]
\[\rho\mapsto \rho \circ F_{0|\hat P_{i+1}\times \mathbb R}\]
is a ring morphism that satisfies:
\[\phi_i^t(I_{i}^t)\subset I_{i+1}^t.\]
Also, the morphism $\phi_i^t$ induces on the quotient $R'_{i}\cong R_{i}^t/I_{i}^t\rightarrow R'_{i+1}\cong R_{i+1}^t/I_{i+1}^t$ the following ring morphism:
\[ \phi'_i:\; R'_{i}\rightarrow R'_{i+1}\]
\[\rho\mapsto \rho\circ f.\]
Let $\tilde R_i$ be the ring $C^\infty (\hat P_i\times \mathbb R)$ formed by the smooth real functions on $\hat P_i\times \mathbb R$.

Let $X$ be in the space of smooth deformations $F:\; \coprod_{i=2}^s \hat P_i \times \mathbb R\rightarrow \coprod_{i=1}^s \hat P_i \times \mathbb R$ of $f_{|\coprod_{i=2}^s \hat P_i}$ endowed with the Whitney topology.

Let $R_i^X$ be the ring $C^0_{F_0} (X,C^\infty(\hat P_i\times \mathbb R))$ formed by the germs at $F_{0|\coprod_{i=2}^s \hat P_i\times \mathbb R}$ of continuous maps from 
$X$ into the space of 
smooth real functions on $\hat P_i\times \mathbb R$.
We notice that $\tilde R_i$ is isomorphic to the quotient $R_i^X/I_i^X$ where $I_i^X$ denotes the ideal formed by the germs that vanish at $F_{0|\coprod_{i=2}^s \hat P_i\times \mathbb R}$. 
For $i< s$, we notice that the map
\[\phi_i^X:\; R_i^X\rightarrow R_{i+1}^X\]
\[(\rho_F)_{F\in X}\mapsto (\rho_F\circ F_{|\hat P_{i+1}\times \mathbb R})_{F\in X}\]
is a ring homomorphism that satisfies 
\[\phi_i^X(I_{i}^X)\subset I_{i+1}^X.\]
Also, the homomorphism $\phi_i^X$ induces on the quotient $\tilde R_i\cong R_i^X/I_i^X\rightarrow \tilde R_{i+1}\cong R_{i+1}^X/I_{i+1}^X$ the following ring homomorphism:
\[\tilde \phi_i:\; \tilde R_i\rightarrow \tilde R_{i+1}\]
\[\rho\mapsto \rho\circ F_{0|\hat P_{i+1}\times \mathbb R}.\]
Let us give the formulation of the problem in these algebraic settings.
For $i\in \{1,\dots, s\}$, we denote by:\begin{itemize}
\item $N_i'$ the $R_i'$-module $\chi^\infty(\hat P_i)$ of smooth vector fields on $\hat P_i$.
\item $N_i^t$ the $R_i^t$-module $\chi_{\hat P_i\times \{t\}}^\infty(\hat P_i\times \mathbb R)^0$ of smooth germs of vector fields on $\hat P_i\times \mathbb R$ at $\hat P_i\times \{t\}$ with $\mathbb R$-component equal to zero. 
\item $\tilde N_i$ the $\tilde R_i$-module $\chi^\infty(\hat P_i\times \mathbb R)^0$ of smooth vector fields on $\hat P_i \times \mathbb R$ with $\mathbb R$-component equal to zero.
\item $N_i^X$ the $R_i^X$-module $C^0_{F_0}(X,\chi^\infty(\hat P_i\times \mathbb R)^0)$ of germs at $F_{0|\coprod_{i=2}^s \hat P_i\times \mathbb R}$ of continuous functions from $X$ into $\chi^\infty (\hat P_i\times \mathbb R)^0$.\end{itemize}
For $i\in\{2,\dots , s\}$, let us denote by:
\begin{itemize}
\item 
$M_i'$ the $R_i'$-module $\chi^\infty (f_{|\hat P_{i}})$ of smooth vector fields along $f_{|\hat P_{i}}$ ({\it i.e.} of smooth sections of $f^*_{|\hat P_i}T\hat P_{i-1}$),
\item $M_i^t$ be the $R_i^t$-modules $\chi^\infty_{\hat P_i\times \{t\}}(F_{0|\hat P_i\times \mathbb R})^0$ of germs at $\hat P_i\times \{t\}$ 
of vector fields along $F_{0|\hat P_i\times \mathbb R}$ with $\mathbb R$-component equal to zero. 
\item $\tilde M_i$ be the $\tilde R_i$-module $\chi^\infty (F_{0|\hat P_i\times \mathbb R})^0$ of smooth vector fields along 
$F_{0|\hat P_i\times \mathbb R}$ with $\mathbb R$-component equal to zero.
\item $M_i^X$ be the $R_i^X$-module of germs at $x_0$ of continuous sections $\sigma$ of the trivial bundle $X\times \chi^\infty (\hat P_i\times \hat P_{i-1}\times \mathbb R)\rightarrow X$, such that  $\sigma(F)$ belongs to $\chi^\infty(F_{|\hat  P_i\times \mathbb R})^0$.  
\end{itemize}
For $i\in \{2,\dots, s\}$, the following maps are homomorphisms of respectively $R_i'$, $R_i^t$, $\tilde R_i$ and $R_i^X$-modules:
\[\begin{array}{cc}
\alpha'_{i,i}:\; N_i'\rightarrow M_{i}'& \alpha^t_{i,i}:\; N_i^t\rightarrow M_{i}^t\\
\xi\mapsto Tf\circ \xi & \xi\mapsto TF^0\circ \xi\\
&\\
\tilde \alpha_{i,i}:\; \tilde N_i\rightarrow \tilde M_{i}&\alpha^X_{i,i}:\; N_i^X\rightarrow M_{i}^X\\
\xi\mapsto TF^0\circ \xi & \quad(\xi_F)_F\mapsto (TF\circ \xi_F)_{F\in X}
\end{array}\]
Also for $i\in \{1,\dots, s-1\}$, the following maps are module homomorphisms over $\phi_i'$, $\phi_i^t$, $\tilde \phi_i$ and $\phi_i^X$ respectively:

\[\begin{array}{cc}
\alpha'_{i,i+1}:\; N_i'\rightarrow M_{i+1}' & \alpha^t_{i,i+1}:\; N_i^t\rightarrow M_{i+1}^t\\
\xi\mapsto \xi\circ f_{|\hat  P_{i+1}} & \xi\mapsto \xi\circ F^0_{|\hat P_{i+1}\times \mathbb R}\\
&\\
\tilde \alpha_{i,i+1}:\; \tilde N_i\rightarrow \tilde M_{i+1} & \alpha^X_{i,i+1}:\; N_i^X\rightarrow M_{i+1}^X\\
\xi\mapsto \xi\circ F^0_{|\hat P_{i+1}\times \mathbb R} & \quad (\xi_F)_F\mapsto (\xi_F\circ F_{|\hat  P_{i+1}})_{F\in X}
\end{array}\]

For $\delta\in \{',t,\tilde\;,X\}$, we denote by 
\[M^\delta:=\bigoplus_{i=2}^s M_i^\delta\qquad \mathrm{and}\qquad N^\delta:=\bigoplus_{i=1}^s N_i^\delta\]
the Abelian sum of the modules $(M_i^\delta)_{i=2}^s$ and $(N_i^\delta)_{i=1}^s$ respectively. The modules $M^\delta$ and $N^\delta$ are modules over the rings 
$A:=\bigoplus_{i=2}^s R_i^\delta$ and
$B:=\bigoplus_{i=1}^s R_i^\delta$ respectively.
Let $I\cdot M^\delta$ (resp. $I\cdot N^\delta$) be the submodule of $M^\delta$ (resp. $N^\delta$) spanned by elements of the form $i_k\cdot x_k$, with $i_k\in I_k^\delta$, $x_k\in M_k^\delta$ (resp. $x_k\in N_k^\delta$) and $k\in \{2,\dots ,s\}$ (resp. $k\in \{1,\dots ,s\}$).
Let us consider the (additive) group morphism:
\[\alpha^\delta:\; N^\delta\rightarrow M^\delta\]
\[(\xi_i)_{i=1}^s\mapsto \big(\alpha_{i,i}^\delta(\xi_i)-\alpha_{i-1,i}^\delta(\xi_{i-1})\big)_{i=2}^s.\]
Let us show that the infinitesimal hypothesis implies the surjectivity of the map $\alpha'$:
for every $\zeta\in \chi^\infty (f_{|\coprod_{i=2}^s \hat P_i})$, we can find a smooth function $r:\; \coprod_{i=1}^s \hat P_i\rightarrow \mathbb (0,+\infty)$, which is $f$-invariant ($\forall x\in \coprod_{i=2}^s \hat P_i$, $r\circ f(x)=r(x)$) and such that $r\cdot \zeta$ can be extended to a smooth section $\zeta'$ of $f^*T\mathcal L$. 
Let $\xi'$ be a vector field on $\mathcal L$ such that:
\[\xi'\circ f-Tf\circ \xi'=\zeta'.\]
Let $\xi'':=-\frac{1}{r}\cdot \xi'_{|\coprod_{i=1}^s \hat P_i}$. We notice that $\xi'' $ is a smooth vector field on $\coprod_{i=1}^s \hat P_i$. Also:
\[Tf\circ \xi''-\xi''\circ f=\frac{1}{r}(\xi'\circ f-Tf\circ \xi')=\frac{\zeta'}{r}=\zeta, \quad \mathrm{on\;} \coprod_{i=2}^s \hat P_i.\]
Let us show that the proposition is proved if we show that $\alpha^X$ sends $I^X\cdot N^X$ onto $I^X\cdot M^X$.

For $F$ in $X$, let $\tau_{F}:=T\pi\circ TF\circ \frac{\partial}{\partial t}$.
We notice that $(\tau_{F})_{F\in X}$ is an element of $I^X\cdot M^X$. 
Thus, we have the existence of $(\xi_F)_{F\in X}\in I^X\cdot N^X$ and of a neighborhood $X'$ of $F_{0|\coprod_{i=2}^s \hat P_i}$ in $X$ such that 
for every $F\in X'$, we have on $\coprod_{i=2}^s\hat P_i$:
\[\tau_{F}:=TF\circ \xi_F-\xi_F\circ F.\]

Let us construct $(U_i)_i$. Let $\hat  T$ be a locally compact metric space such that a small neighborhood $\hat U_1$ of $\hat P_1$ is isomorphic to the product $\hat P_1\times \hat T$. 
Let $\tau_0\in \hat T$ s.t. this isomorphism sends $\hat P_1$ to $\hat P_1\times \{\tau_0\}$. 
Since $f$ is transversally bijective, for $\hat U_1$ small enough, we notice that a neighborhood $\hat U_i$ of $\hat P_i$ is canonically isomorphic to $\hat P_i\times \hat T$, and this isomorphism sends $\hat P_i$ onto $\hat P_i\times \{\tau_0\}$. For $\hat T$ sufficiently small the open subsets $(\hat U_i)_i$ are disjoint.
Let $\hat \Omega:=\coprod_{i=1}^s \hat U_i$ and $\hat \Omega':=\coprod_{i=2}^s \hat U_i$. 
Let $F_0: L\times \mathbb R\rightarrow L\times \mathbb R$ be the trivial deformation of $f$.

Let $\exp$ be the exponential map associated to a complete metric on $\coprod_{i=1}^s \hat P_i$. Let $r\in C^\infty (\coprod_{i=1}^s \hat P_i\times \mathbb R, \mathbb R)$ be a compactly supported function equal to 1 on a neighborhood of $\coprod_{i=1}^s P_i\times B$. There exists a $CO^\infty$-neighborhood $V_F$ of the trivial deformations $F^0$ of $f$ and a neighborhood $T$ of $\tau_0\in \hat T$ such that for every deformation $F\in V_F$ and $\tau\in T$ the following map is well defined:
\[F_\tau:\; \coprod_{i=2}^s \hat P_i\times \mathbb R\rightarrow \coprod_{i=1}^s \hat P_i \times \mathbb R\]
\[ (x,t)\mapsto \left[ \begin{array}{cl} \Big(\exp_{f(x,\tau_0)}\big( r(x,t)\cdot \exp^{-1}_{f(x,\tau_0)} \big(F(x,\tau,t)\big)\big), t\Big)& \mathrm{if}\; r(x,t)\not=0\\
   \big( f(x,\tau_0), t)& \mathrm{else}\end{array}\right.\]    
with in particular the restriction of $F$ to $\coprod_{i=2}^s P_i\times \{\tau \}\times \mathbb R$ canonically identitified to a map from $\coprod_{i=2}^s P_i\times \mathbb R$ into $\coprod_{i=1}^s \hat P_i \times \mathbb R$. Also, when $F$ is $CO^\infty$-close to $F_0$ and $\tau $ is close to $\tau_0$, then $F_\tau$ is $\mathcal W^\infty$-close to $F_{0|\coprod_{i=2}^s \hat P_i\times \mathbb R}$. We suppose $V_F$ and $T$ sufficiently small such that  $F_\tau$ belongs to $X'$ and such that $r\circ F_\tau$ is equal to 1 on $\coprod_{i=2}^s P_i\times B$, for every $F\in V_F$ and $\tau \in T$.  

Let $\rho\in C^0(\hat T,\mathbb R)$ be a function equal to $1$ on a neighborhood $T'$ of $\tau^0\in T$ and to 0 off $T$. For $F\in V_F$, let $U_i:= P_i\times T'$, $\Omega:= \coprod_{i=1}^s U_i$, $\Omega':= \coprod_{i=2}^s U_i$ and:
\[\xi: = z\in L\times \mathbb R\mapsto \left\{\begin{array}{cl} \rho(\tau)\cdot r(x,t)\cdot \xi_{F^\tau}(x,t) & \mathrm{if} \; z= (x,\tau, t)\in \coprod_{i=1}^s \hat P_i \times T\times \mathbb R\\
0& \mathrm{else}\end{array}\right..\]
We notice that $\xi$ belongs to $\chi^\infty (\mathcal L\times \mathbb R)^0$ and that we have on $\Omega'$:
\[\tau_F:= T\pi \circ TF\circ \frac{\partial}{\partial t} = TF\circ \xi -\xi\circ F.\]

Also when $F$ is $CO^\infty $-close to $F_0$, then $\xi$ is $\mathcal W^\infty$-small.
Hence the proposition is shown.

The proof that the surjectivity of all $(\alpha^t)_{t\in \mathbb R}$, implies the surjectivity of $\tilde \alpha$ is easy. It will be done at the end.

To show that the surjectivity of $\alpha'$ implies the one of $\alpha^t$, and that the one of $\tilde \alpha$ implies the one of $\alpha^X$ and that $\alpha^X(I^X\cdot N^X)=I^X\cdot M^X$, we shall use the following techniques of Mather.

\paragraph{ The algebraic Machinery}
Let $R$ and $S$ be rings with units. Let $I$ and $J$ be \emph{Jacobson ideals} (this means that for every $z\in J$, the element $1+z$ is invertible) in $R$ and $S$ respectively. Let 
$\phi:\; R\rightarrow S$
be a ring homomorphism which sends $I$ into $J$.
\begin{defi} The homomorphism $\phi:\; (R,I)\rightarrow (S,J)$ is \emph{adequate} if the following condition is satisfied:
Let $A$ be a finitely generated $R$-module. Let $B$ and $C$ be $S$-modules, with $C$ finitely generated over $S$.
Let $\beta:\; B\rightarrow C$ be a homomorphism of $S$-modules.
Let $\alpha:\; A\rightarrow C$ be a homomorphism over $\phi$ ({\it i.e.} 
$\alpha(a+b)=\alpha(a)+\alpha(b)$ and $\alpha(r\cdot a)=\phi(r)\cdot \alpha(a)$, for $a,b\in A$ and $r\in R$).
Suppose that:
\[\alpha(A)+\beta(B)+J\cdot C=C.\]
Then we can conclude
\[\alpha(A)+\beta(B)=C\qquad \mathrm{
and}\qquad  
\alpha(I\cdot A)+\beta(J\cdot B)=J\cdot C.\]
\end{defi}
Let us illustrate the above definition by the following non-trivial examples shown by Mather in \cite{Ma1}-\cite{Ma2}-\cite{Ma3} and rewritten in this form by Tougeron \cite{Tou}:
\begin{thm}\label{ade1}
For any $i\in \{2,\dots , s\}$, the ring homomorphisms:
\[\phi_i^t:\; (R_i^t,I_i^t)\rightarrow (R_{i+1}^t, I_{i+1}^t)\]
and
\[\phi_i^X:\; (R_i^X,I_i^X)\rightarrow (R_{i+1}^X, I_{i+1}^X)\]
are adequate.\end{thm}
This is the algebraic theorem of Mather:
\begin{thm}[Mather] Let $(R_i,I_i)$, $i=1,\dots , s$ be rings with units where $I_i$ is a Jacobson ideal for every $i$.
Let 
\[(R_1,I_1)\stackrel{\phi_1}{\rightarrow} (R_2,I_2)\stackrel{\phi_2}{\rightarrow}\cdots \stackrel{\phi_{s-1}}{\rightarrow}(R_s,I_s)\]
be a sequence of adequate homomorphisms.
For every $i$, let $N_i$ and $M_i$ be $R_i$-modules, finitely generated (with possible exception for $N_s$). Put for $i<j$:
\[\phi_{ij}= \phi_j\circ \cdots \circ \phi_i\]
and $\phi_{ii}$ the identity of $R_i$.
For $j\ge i$, let $\alpha_{ij}:\; N_i\rightarrow M_j$ be a module-homomorphism over $\phi_{ij}$. 
Let $N:=\oplus_{i=1}^s N_i$ and $M:=\oplus_{i=1}^s M_i$ as the direct sums of Abelian groups and let 
\[\alpha:\; N\rightarrow M\]
be given by 
\[\alpha(\xi_1,\dots ,\xi_s)=
(\alpha_{11}(\xi_1), 
\alpha_{12}(\xi_{1})+\alpha_{22}(\xi_2), \dots,
\alpha_{1s}(\xi_1)+\dots + \alpha_{ss}(\xi_s).\]
Suppose that 
\begin{itemize}\item[(I)]
\qquad\qquad$\alpha(N)+\sum_{i=1}^sI_i M_i=M.$
\end{itemize}
Then:
\begin{itemize}\item[(II)]
\qquad\qquad $\alpha$ sends $N$ onto $M$.
\item[(III)]\qquad\qquad
Moreover $\alpha$ sends $\sum_{i=1}^s I_i\cdot N_i$ onto $\sum_{i=1}^s I_i\cdot M_i$.\end{itemize} 
\end{thm}
\begin{rema}
We can illustrate the morphism $\alpha$ by the following diagram:
\[\begin{array}{rcccccl}
(R_1,I_1)&\stackrel{\phi_1}{\rightarrow} &(R_2,I_2)&\stackrel{\phi_2}{\rightarrow}&\cdots &\stackrel{\phi_{n-1}}{\rightarrow}&(R_s,I_s)\\
&&&&&&\\
N_1 &\oplus &N_2 &\oplus &\cdots & \oplus & N_s \\
&&&&&&\\
\alpha_{11}\downarrow&\alpha_{12}\searrow& \alpha_{22}\downarrow& & \searrow\cdots \searrow&\searrow&\downarrow\alpha_{ss}\\
&&&&&&\\
M_1&\oplus&M_2&\oplus & \cdots &\oplus & M_s.\end{array}
\]
\end{rema}
\begin{rema} For $s=1$ this theorem is the Nakayama's lemma. For $s=3$ is was shown by F. Latour.\end{rema}
As the proof is purely algebraic, we will prove this theorem at the end of this work. Let us conclude the proof of the proposition.

If we omit the exponent $X,t,\;'\;,\tilde\;$, and we put $M_1=0$ and $\alpha_{ji}=0$ when $j< i-1$, it follows from the last theorem that the surjectivity of $\alpha'$ implies the one of $\alpha^t$, and that the surjectivity of $\tilde \alpha$ implies the one of $\alpha^X$ with 
$\alpha^X(I^XN^X)=I^XM^X$.

Thus it only remains to prove that the surjectivity of all $(\alpha^t)_{t\in \mathbb R}$ implies the one of $\alpha'$.
Let $\tau\in \chi^\infty (F_0)^0$. 
For $t\in \mathbb R$, as $\alpha^t$ is surjective, there exists a germ $\xi_t\in A_t$ such that $\tau_t:=\alpha_t(\xi_t)$. The germs $\xi_t$ is defined on a neighborhood $V_t$ of $\coprod_{i=1}^s \hat P_i\times \{t\}$ in $\coprod_{i=1}^s \hat P_i\times \mathbb R$. 
By shrinking a slice $\hat \Omega$ and then by shrinking $V_t$ for every $t\in \mathbb R$, we may suppose that $V_t$ is of the form $(\coprod_{i=1}^s \hat P_i)\times W_t$, where $W_t$ is a neighborhood of $t\in \mathbb R$. 
Let $(W_j)_{j\in \mathbb N}$ be a locally finite subcovering of $(W_t)_{t\in \mathbb R}$.
By locally finite we mean that there exists for every point $t\in \mathbb R$ a finite number of integers $j\in \mathbb N$ such that $W_{j}$ intersects a neighborhood of $t$. 
By subcovering we mean that for every $j\ge 0$ there exists $t_j$ s.t. $W_j$ is included in $W_{t_j}$. 
Let $(\rho_j)_j$ be a partition of the unity subordinate to $(W_j)_j$.
Let $\pi':\; \coprod_{i=1}^s \hat P_i\times \mathbb R\rightarrow \mathbb R$ be the projection on the second coordinate. We notice that:
\[\pi' \circ F_0=\pi'\quad \mathrm{on}\; \coprod_{i=2}^s \hat P_i\times \mathbb R,\]
since the map $F_0$ is a deformation of $f$.
Let $\xi:= \sum_{j=1}^\infty \rho_j\circ \pi'\cdot \xi_{t_i}$.
We notice that $\xi$ is sent by $\alpha'$ to $\tau$. This concludes the proof of the proposition.
\end{proof}

\begin{rema}\label{rem pour 1.7} We notice that one easily simplifies the above proof to show that under the hypotheses of Theorem \ref{1.7} of Mather, for every deformation $(F_{ij})_{[i,j]\in A}$ which is $\mathcal W^\infty$-close to the trivial deformation of $(f_{ij})_{[i,j]\in A}$ there exists a $\mathcal W^\infty$-small vector vector field $\xi\in \chi^\infty(\coprod_{i\in V} M_i\times \mathbb R)^0$ such that 
\[(TF_{ij}\circ \xi -\xi\circ f_{ij})_{|M_i}= TF_{ij}\circ \frac{\partial}{\partial t} - \frac{\partial}{\partial t} \circ F_{ij|M_i},\]
for every $[i,j]\in A$.

Thus by using  Remark \ref{pour 1.7 aussi} with $C$ the disjoint union of the manifold from which an arrow start, we get the proof of the Theorem \ref{1.7}.
\end{rema}
\subsection{From local to global: proof of Theorem $\ref{fonda}$}

Let $f$ be an endomorphism of $\mathcal L$ and  let $C$ be a compact subset of $L$  as in the statement of the theorem. 
We notice that by compactness of $C$, there exists a compact neighborhood $V$ of $K$ s.t. for every $x\in V$ some iterate $f^n(x)$ does not belong to $V$, for $n\in \mathbb N$. Let $W:= f(V)\setminus int(V)$. For every $p\in  W$, we define $(p_i)_{i\ge 1}$ inductively: $P:=\{p\}$; $P_{i+1}:= f^{-1}(P_i)\cap V$. We notice that  $P_N$ is empty. Let $s_p$ be maximal s.t. $P_{s_p}$ is not empty. We define $(\hat P_i^p)_{i=1}^{s_p}$ inductively: $\hat P_1^p$ is a small plaque that contains $p$ and $\hat P_{i+1}^p:= f^{-1}(\hat P_1^p)$. We notice that for $\hat P_1^p$ sufficiently small, $\hat P_i^p$ is disjoint from $\hat P_1^p$ (else we would have a cycle in $V$). Let $P_1^p$ be a precompact neighborhood of $P_1^p$ and let  $(U_i^p)_{i=1}^{s_p}$ be the open subsets provided by Proposition \ref{local:version}.

 Also, by shrinking if necessarily, we may suppose that the closure of $U_1^p$ does not intersect $C$.
As $W$ is compact, we can extract from $(U_1^p)_{p\in C}$ a finite subcover $( U^{p_j}_1)_{j\ge 1}$ of $W$.
As $W$ is compact and $\cap_{k=0}^N f^{-n}(V)$ is empty, we notice that for every $f'$ $CO^\infty$-close enough to $f$, for every $x\in C$, there exists an integer $n\ge 1$, s.t.: 
\[f^n(x)\in \Delta:=\cup_{j\ge 1} \breve U_1^{p_j}.\]

For every $CO^\infty$-small deformation $F$ of $f$, for every $j$, there exists  $\xi_j\in \chi^\infty (\cup_{i=1}^{s_j} U_i^{p_j}\times \mathbb R)^0$, such that 
\[\tau_{F}= TF\circ \xi_j- \xi_j\circ F,\quad  \mathrm{on}\; \cup_{i=2}^{s_j}U_i^{p_j}\times \mathbb R\]
Let $(\rho_1^j)_{j=1}^N$ be a partition of the unity subordinate to $(U_1^{p_j})_{j}$. 
\[\mathrm{Let}\; \rho_2^j= (x,t)\in L\times \mathbb R\mapsto \left\{ \begin{array}{cl}
\rho_1^j\circ p_1\circ F^i(x,t)& (x,t)\in U_i^{p_j}\times \mathbb R\\
0 & \mathrm{else}\end{array}\right.,\]
with $p_1:\; L\times \mathbb R\rightarrow L$ the canonical projection.

Let $R:= \sum_{j=1}^N \rho_2^j$. Since $(U_i^{p_j})_{i,j}$ is finite,
the function $R$ is well defined and is a smooth morphism of $\mathcal L\times \mathbb R$ into $\mathbb R$.
Let $\rho_j:= \frac{\rho_2^j}{R}$ and $\xi:= x\mapsto \sum_{j=1}^N \rho_j(x)\xi_j(x)$. As $C$ is disjoint from the closure of $\cup_{j=1}^N U_1^{p_j}$ we have for all $(x,t)\in C\times \mathbb R$, 
\[\tau_F(x,t)=TF\circ \xi(x,t)-\xi\circ F(x,t).\] 

\subsection{Infinitesimal stability for manifold implies infinitesimal stability for embedded lamination}\label{the section}
We recall that to prove Theorem \ref{fonda}, we only used the surjectivity of the map:
\[\chi^\infty (\coprod_{i=1}^s P_i)\rightarrow \chi^\infty (f_{|\coprod_{i=2}^s P_i})\]
\[\sigma\mapsto \sigma\circ f-Tf\circ \sigma.\]
But this surjectivity is an easy consequence of the infinitesimal stability of $f$ stated in hypothesis $(ii)$:
for every $\tau\in \chi^\infty (f_{|\coprod_{i=2}^s P_i})$, we can find a smooth function $r:\; \coprod_{i=1}^s P_i\rightarrow \mathbb (0,+\infty)$, which is $f$-invariant ($\forall x\in \coprod_{i=2}^s P_i$, $r\circ f(x)=r(x)$) and such that $r\cdot \tau$ can be extended to a smooth section $\tau'$ of $f^*TM$. 
Let $\xi'$ be a vector field on $M\setminus \hat \Omega$ such that on this domain:
\[\xi'\circ f-Tf\circ \xi'=\tau'\]
Let $\xi'':=\frac{1}{r}\cdot \xi'_{|\coprod_{i=1}^s P_i}$. We notice that $\xi'' $ is a smooth vector field on $\coprod_{i=1}^s P_i$. Also
\[\xi''\circ f -Tf\circ \xi''=\frac{1}{r}(\xi'-Tf\circ \xi')=\frac{\tau'}{r}=\tau\]
Now we have to transform $\xi'' $ to a vector field tangent to $\coprod_{i=1}^s P_i$.
Let $N_1\rightarrow P_1$ be the smooth vector bundle whose fiber at $x\in P_1$ is $T_xP_1^\bot$.
Let $N_i\rightarrow P_i$ be the smooth vector bundle whose fiber at $x\in P_i$ is $T_xf^{-i+1}(T_{f^{i-1}(x)}P_1^\bot)$.
By transversality of $f$ to $\mathcal L$, $P_1$ is also a smooth vector bundle.
Let $\pi$ be the projection of $TM_{|\coprod_{i=1}^s P_i}$ onto $T(\coprod_{i=1}^s P_i)$ parallelly to $\coprod_i N_i\rightarrow \coprod_{i=1}^s P_i$.
We notice that $\xi:= \pi\circ \xi'' $ satisfies the requested properties.

\subsection{Proof of Mather's algebraic theorem}
The proof of this theorem comes from an unpublished manuscript of Mather. It was then rewritten by Baas in 
a manuscript as well unpublished. Here we only copy the last proof.
The first step is to show that it is sufficient to prove the theorem for $M_1=0$. Let us assume that this has been proved and from this prove the theorem.
Put
\[ M'=\bigoplus_{i=2}^{s} M_i\]
and 
\[\alpha'\;:\; N\rightarrow M'\]
is given by 
\[\alpha'(n_1,\dots , n_s)=(\alpha_{12}(n_1)+\alpha_{22}(n_2),\dots , \alpha_{1s}(n_{1})+\cdots + \alpha_{ss}(n_s)).\]
The hypothesis $(I)$ clearly implies:
\begin{itemize}\item[(I')]
\qquad\qquad$\alpha'(N)+\sum_{i=2}^{s}I_iM_i=M'.$
\end{itemize}
Then the special case of the theorem with $M_1=0$ which we assume gives: 
\[\alpha'(N)=M'\quad \mathrm{and}\quad \alpha'\Big(\sum_{i=1}^s I_i\cdot N_i\Big)= \sum_{i=2}^{s} I_i\cdot M_i.\]
\[\mathrm{Let} 
\quad N_1'=(\alpha_{12}+ \cdots + \alpha_{1,s})^{-1}\alpha\Big(\sum_{i=2}^{s} N_i\Big).\]
It is now sufficient to prove 
\[ \alpha_{11}(N'_1)=M_1.\]
This would give $(I)$ and also imply:
\[\alpha_{11}(I_1N'_1)=I_1M_1\]
which together with
\[(\alpha_{12}+\cdots +\alpha_{1s})(I_1\cdot N_1')\subset \alpha\Big(\sum_{i=2}^{s} I_i\cdot N_i\Big)\]
gives $(III)$.
Let us prove that 
\[\alpha_{11}(N_1')=M_1\]
follows from this implication of $(III')$: 
\[M=\alpha(N)+\alpha'\Big(\sum_{i=1}^s I_i\cdot N_i\Big)+I_1 \cdot M_1=
\alpha\Big(\sum_{i=2}^{s}
N_i\Big) +\alpha'(I_1\cdot N_1)+ \alpha(N_1)+I_1\cdot M_1.\]
For $m_1\in M_1$, there exist $n_1'\in I_1\cdot N_1$, $m_1'\in I_1\cdot M_1$ and $n_i\in N_i$, for $i\in \{1,\dots , s\}$, such that:
\[m_1\oplus 0\oplus\cdots \oplus 0=
\alpha(n_2+\cdots + n_{s})+\alpha'(n_1')+\alpha(n_1)+(m_1'\oplus 0\cdot \oplus 0).\]
Componentwise this gives:
\begin{eqnarray}
b_1=\alpha_{11}(n_1)+m_1'\\
0=(\alpha_{1s}+\cdots +\alpha_{2s})(n_1)+(\alpha_{12}+\cdots +\alpha_{1s})(n_1')+\alpha(n_2+\cdots +n_{s}).
\end{eqnarray}
From the second equation, it follows that the point $n_1+n_1'$ belongs to $N_1'$. Hence:
\[m_1=\alpha_{11}(n_1+n_1')+(m'_1-\alpha_{11}(n'_1))\in \alpha_{11}(N'_1)+I_1M_1.\]
So we have shown
\[ M_1=\alpha_{11}(N'_1)+I_1M_1\]
and by applying Nakayama's Lemma, we get:
\[\alpha_{11}(N'_1)=M_1.\]
And this finishes the proof of the theorem, assuming its holds for $M_1=0$.
So the next step is to prove the theorem for $M_1=0$. This is done by induction on $s$. For $s=1$, the theorem is trivial.
So assume the theorem inductively for $s-1$.\\
Put 
\[N_1^*=N_1\bigotimes_{R_1} R_{2}\]
where $R_{2}$ is regarded as an $R_1$-module via $\phi_{1}$. 
Let 
\[\alpha_{1j}^*:= \alpha_{1j}\otimes\phi_{2j}:\; N_1^*\rightarrow M_j,\; j\ge 2\]
and 
\[N^*= N_2\oplus \cdots \oplus N_{s}\oplus N_1^*\]
\[M=M_2\oplus \cdots \oplus M_{s}\]
Define 
\[\alpha^*:\; N^*\rightarrow M\]
by
\[\alpha^*(n_1,\dots ,n_s)=\Big(
\alpha_{22}(n_{2})+\alpha_{12}^*(n_1),
\dots ,
\alpha_{1s}^*(n_1)+\cdots +\alpha_{ss}(n_s)
\Big)\]
Clearly $N_1^* $ is a finitely generated $R_{2}$-module and $\alpha_{1j}^*$ is a homomorphism over $\phi_{1j}$.
Now $(I)$ implies 
\[\alpha^*(N^*)+\sum_{i=2}^{s}I_iM_i=M.\]
And by the induction hypothesis we conclude
\begin{itemize}\item[$(II^*)$]
\qquad\qquad $\alpha^*$ sends $N^*$ onto $M$.
\item[$(III^*)$]\qquad\qquad
Moreover $\alpha^*$ sends $\sum_{i=2}^s I_i\cdot N_i+ I_{2}N_1^*$ onto $\sum_{i=2}^s I_i\cdot M_i$.\end{itemize} 
Let 
\[\beta =id \otimes\phi_{1}:\; 
N_1=N_1\bigotimes_{R_1} R_1\rightarrow N_1^*= N_1\bigotimes_{R_1} R_{2}.\]
Then $\alpha^*_{1j}\circ \beta$ is equal to    $\alpha_{1j}$, for every $j$. .
Let 
\[C:= (\alpha_{12}^*+\cdots +\alpha_{1s}^*)^{-1}\cdot \alpha\Big(\sum_{i=2}^{s} N_i\Big)\]
and this is an $R_{2}$-submodule of $N_1^*$. Let $n^*\in M_1^*$ and $m=\alpha^*(n^*)\in M$. Then by our assumption $(I)$:
\[m=\alpha(n_1+\cdots +n_s)+m_2'+\cdots +m_{s}',\]
where $n_i\in N_i$ and $m_i'\in I_i\cdot M_i$.
By $(III^*)$
\[m_2'+\cdots +m_{s}'=\alpha^*(n_1^*+\cdots +n_s^*)\]
where $n_1^*\in I_{2}\cdot  N_1^*$, and $n_i^*\in I_i N_i$, for $i\ge 2$. Hence 
\[m=\alpha\Big( (n_1+(n_2+n_2^*)+\cdots +(n_{s}+n_{s}^*)\Big)+\alpha^*(n_1^*).\]
Therefore
\[\alpha^*(n^*-\beta(n_1)-n_1^*)=m-\alpha(n_1)-\alpha^*(n_1^*)=
\alpha\big((n_2+n_2^*)+\cdots + (n_{s}+n_{s}^*)).\]Put 
\[c=n^*-\beta(n_1)-n^*_1\in C\]
Since 
\[n^*=c+\beta(n_1)+n_1^*\in C+\beta(N_1)+I_{2}N_1^*\]
We get \[ N_1^*=C+\beta(N_1)+I_{2}N_1^*.\]
Now since $\phi_{2}$ is adequate we deduce
\begin{itemize}\item[$(II^{**})$]
\qquad\qquad $N_1^*=C+\beta(N_1)$.
\item[$(III^{**})$]
\qquad\qquad $I_{2}N_1^*=I_{2}C+\beta(I_1N_1)$.
\end{itemize}
Clearly $(II^*)$ and $(II^{**})$ give $(II)$ and $(III^*)$. Also $(III^*)$ and $(III^{**})$ give $(III)$. This finishes the proof of the theorem.
\begin{flushright}
$\square$
\end{flushright}%\subsection{Condition $\mathcal O\Rightarrow$ condition $\mathcal I$}
\subsection{Proof of Lemma \ref{trivial inf stab}}
\label{final lemma}
Let $\xi\in \chi^\infty(f_{|M'})$. 
By $(c)$ there exists $\sigma_0\in \chi(M')$ such that:
\[(\sigma_0\circ f-Tf\circ \sigma_0)_{|U}= \xi_{|U}.\]
By restricting a slice $U$ s.t. $(a)$ and $(b)$ are still satisfied, we may suppose that  $\sigma_0$ can be smoothly extended to $M$.  
  
      We define inductively $(\sigma_n)_{n\ge 0}\in \prod_n \chi(f_{|\cup_{k=0}^n f^k(U)})$ by:
      \[ \sigma_{n+1}:x\mapsto \left\{\begin{array}{cl} \sigma_n(x)& \mathrm{if}\; x\in \cup_{k=0}^n f^k(U)\\
      \xi\circ f^{-1}(x)+Tf\circ \sigma_n\circ f^{-1}(x)& \mathrm{if} \; x\in f^{n+1}(U)\setminus U\end{array}\right.\]
      We notice that $(\sigma_n)_n$ is well defined and is locally eventually constant on the open subset $U^+:= \cup_{k\ge 0} f^k(U)$. Thus $(\sigma_n)_n$ converges to some section $\sigma^+_0\in \chi(U^+)$. Also $\sigma^+$ satisfies:
      \[(\sigma^+_0\circ f-Tf\circ \sigma^+_0)_{|U^+}= \xi_{|U^+}.\]
      
       Let us define $(\sigma_n^+)_n\in \prod_{n\ge 0} \chi(f_{|f^{-n}(U^+)})$ by  induction.  Let $n\ge 0$ and suppose $\sigma_n^+$ constructed.
       For this end we notice that the forward orbit $O^+(\Sigma)$ of the singularities of $f$ is closed and that on $M'':=M'\setminus O^+(\Sigma)$, the map $S:=x\mapsto \ker(T_xf)$ is a smooth section of the Grassmannian of $TM''$. Let $p$ be the orthogonal projection of $TM''$ onto $S$. Remember that $\sigma_0$ can be smoothly extended to $M'$. This implies that $\sigma_n$ can be smoothly extended to $M'$. Let $\sigma_n^s\in \chi(M')$ be a smooth extension of $p\circ \sigma_{n|M''}^+$ such that $\sigma_n^s(x)$ belongs to $S(x)$, for every $x\in M'$.
       We can now define inductively for $n\ge 0$:
       \[\sigma_{n+1}^+:= x\in f^{-n-1}(U^+)\mapsto \left\{\begin{array}{cl} 
       \sigma_n^+(x) & \mathrm{if} \; x\in  f^{-n}(U^+)\\
       (Tf^{-1}_{|S(x)^\bot}\big(\sigma_n^+\circ f(x)-\xi(x)\big) & \mathrm{if}\; x\in f^{-n-1}(x)\end{array}\right.\]
       
       We notice that $(\sigma_n^+)_n$ is well defined and eventually constant on $\hat U:= \cup_{n\ge 0} f^{-n}(U^+)$.  Thus $(\sigma_n^+)_n$ converges to some section $\hat \sigma\in \chi(\hat U)$. Also $\hat \sigma$ satisfies:
       \[(\hat \sigma\circ f-Tf\circ \hat \sigma)_{|\hat U}= \xi_{|\hat U}.\]
       
Let $U'$ be an open neighborhood of the singularities satisfying Property $(b)$ of the lemma, such that $U$ contains the closure of $U'$.       

       Let $\hat U':=\cup_{n\ge 0} f^{-n}(\cup_{k\ge 0} f^k(U'))$ and let $\breve U$ be the complement of $\hat U'$ in $M'$. We notice that $(\hat U, \breve U)$ is an open cover of $M'$. Thus to finish the proof of the lemma, it is sufficient to find a partition of the unity $(r,1-r)$ subordinate to this cover, which is $f$-invariant ($r\circ f=r)$ and to find a section $\breve \sigma\in \chi(\breve U)$ such that 
       \[(\breve \sigma\circ f-Tf\circ \breve \sigma)_{|\breve U}=\xi_{|\breve U}.\]
        Then we notice that $\sigma:= r\cdot \hat \sigma+(1-r)\cdot \breve \sigma$ satisfies the requested property.

       Let $V$ be a manifold with boundary such that    $A$ is the maximal invariant of $V$ ({\it i.e} $A=\cap_n f^n(V)$) and $f$ sends $V$ into its interior.
       
       Let us construct $r$.
       As we are here in the diffeomorphism case, the construction a partition of the unity $(r_1, 1-r_1)$ subordinate to the cover $(V\cap \hat U, V\cap \breve U)$ and $f$-invariant is classic. Then we define $r:= x\in M'\mapsto r\circ f^n(x)$, if $x\in f^{-n}(V)$ which is convenient for our purpose.
       
      Let us construct $\breve \sigma$. Let $D:= (V\setminus f(V))\cap \breve U$.  Let $\partial^+ D:= \partial D\cap  f(V)\cap \breve U$ and $\partial^- D= \partial D\cap int(V)^c\cap \breve U$, with $\partial D$ the boundary of $D$. On a neighborhood of $\partial^- D$ we define $\breve \sigma_+=0$ and on a neighborhood of $\partial^- D$ we define $\breve \sigma_-=Tf^{-1}\circ (\breve \sigma_+\circ f-\xi)$. Then we chose a section $\breve \sigma_0$ of $\chi^\infty(M)$ equal to $\breve \sigma_+=0$ on a neighborhood of $\partial^+ D$ and to  $\breve \sigma_-$ on a neighborhood of $\partial^- D$.  
       
      As for the construction of  $\hat \sigma$, we define then $\breve \sigma$ on $\breve U^+:= \cup_{n\ge 0} f^n(D)$ and finally on $\breve U= \cup_{n\ge 0} f^{n} (\breve U^+)$.     
\begin{flushright}
$\square$
\end{flushright}

{\footnotesize
\bibliographystyle{alpha}
%\nocite{*}
\bibliography{references}}

\bigskip
Pierre Berger,\\
 Institute for Mathematical Sciences, Stony Brook University, Stony Brook, NY 11794-3660\\
\noindent\url{pierre.berger((at))normalesup.org}\\
\noindent \url{http://www.math.sunysb.edu/~berger/}
\end{document}